\documentclass{article}
\usepackage{lineno,hyperref}
\modulolinenumbers[5]

\usepackage{authblk}
\usepackage{graphicx}
\usepackage{amsmath, amsfonts, amsthm, amssymb, bm}
\usepackage{color}
\usepackage{tabularx,hhline}
\usepackage{multirow}
\usepackage{booktabs}
\usepackage{fullpage}[2cm]
\usepackage{algorithm}
\usepackage{algorithmic}
\usepackage{enumerate}
\usepackage{bbm}
\usepackage[margin = 1.5cm, font = small, labelfont = bf, labelsep = period]{caption}
\usepackage[font = small, labelfont = bf, labelsep = period]{caption}
\usepackage{algorithm}
\usepackage{algorithmic}
\usepackage{verbatim}

\newcommand{\st}{\textnormal{s.t.}}
\newcommand{\conv}{\textnormal{conv}}
\newcommand{\cone}{\textnormal{cone}}

\DeclareMathOperator*{\argmax}{arg\,max}
\DeclareMathOperator*{\argmin}{arg\,min}
\DeclareMathOperator*{\diag}{diag}

\newcommand{\eg}{\textit{e.g.}}
\newcommand{\ie}{\textit{i.e.}}

\newcommand{\RR}{\mathbb{R}}

\DeclareMathAlphabet{\mathscr}{U}{dutchcal}{m}{n}
\SetMathAlphabet{\mathscr}{bold}{U}{dutchcal}{b}{n}
\DeclareMathAlphabet{\mathbscr} {U}{dutchcal}{b}{n}

\newcommand{\soc}{\mathbb S\mathbb O\mathbb C}

\newcommand{\tr}{\textup{tr}}

\newcommand{\cl}{\textup{cl}}

\newtheorem{thm}{Theorem}
\newtheorem{prop}{Proposition}
\newtheorem{lem}{Lemma}

\newtheorem{coro}{Corollary}

\newtheorem{rem}{Remark}

\begin{document}
	%%%%%%%%%%%%%%%%

	%\title{Cluster Analysis is Convex:~New SDP Relaxation and Algorithm for $K$-Means Clustering}
	\title{{Improved Conic Reformulations for K-means Clustering}}
	\author{Madhushini Narayana Prasad}
	\author{Grani A.~Hanasusanto}
	
	\affil{\small Graduate Program in Operations Research and Industrial Engineering, The University of Texas at Austin, USA}
	\date{}	
	\maketitle

	%\affil[3]{\small Imperial College Business School, Imperial College London, United Kingdom}
	%
	%\author[rvt]{C.V.~Radhakrishnan\corref{cor1}\fnref{fn1}}
	%\ead{cvr@river-valley.com}
	%\author[rvt,focal]{K.~Bazargan\fnref{fn2}}
	%\ead{kaveh@river-valley.com}
	%\author[els]{S.~Pepping\corref{cor2}\fnref{fn1,fn3}}
	%\ead[url]{http://www.elsevier.com}
	%%%%%%%%%%%%%%%%%%%%%%%%%%%%%%%%%%%%%%%%%%%%%%%%%%%%%%%%%%%%%%%%%%%%%
	
	\begin{abstract}
		In this paper, we show that the popular $K$-means clustering problem can equivalently be reformulated as a conic program of polynomial size. The arising convex optimization problem is NP-hard, but amenable to a tractable semidefinite programming (SDP) relaxation that  is tighter than the current  SDP relaxation schemes in {the} literature. In contrast to the existing schemes, our proposed SDP formulation gives rise to solutions that can  be leveraged to identify the clusters. We devise a new approximation algorithm for $K$-means clustering that utilizes the improved formulation and empirically illustrate its superiority over the state-of-the-art solution schemes.     
	\end{abstract}
	%In contrast to the existing scheme, our proposed SDP formulation
	\section{Introduction}\label{sec:introduction}
	
	%Given a  set of data points, cluster analysis endeavors to discover a fixed number of disjoint clusters  that are spatially separated and where the data points in each cluster are similar to each other. Cluster analysis is fundamental to a wide array of applications in science, engineering,  economics, psychology, marketing, etc.~\cite{kaufman2009finding,jain2010data}  One of the most popular approaches for cluster analysis is the \emph{$K$-means clustering} \cite{macqueen1967,lloyd1982least,jain2010data}. 
	%Cluster analysis is fundamental to a wide array of applications in science, engineering,  economics, psychology, marketing, etc.~\cite{kaufman2009finding,jain2010data} Given a  set of data points, cluster analysis endeavors to discover a fixed number of disjoint clusters  that are spatially separated and where the data points in each cluster are similar to each other. One of the most popular approaches for cluster analysis is the \emph{$K$-means clustering} \cite{macqueen1967,lloyd1982least,jain2010data}. 
	%The goal of $K$-means clustering is to partition the  data points into $K$ clusters such that the sum of squared {distances} to the respective cluster centroids is minimized. 
	Given an input  set of data points, cluster analysis endeavors to discover a fixed number of disjoint clusters so that the data points in the same cluster are closer to each other than to those in other clusters. Cluster analysis is fundamental to a wide array of applications in science, engineering,  economics, psychology, marketing, etc.~\cite{jain2010data,kaufman2009finding}.  
	One of the most popular approaches for cluster analysis is  \emph{$K$-means clustering}~\cite{jain2010data,lloyd1982least,macqueen1967}. 
	The goal of $K$-means clustering is to partition the  data points into $K$ clusters so that the sum of squared {distances} to the respective cluster centroids is minimized. 
	% which seeks for clusters to minimize the .  
	%A natural measure of this objective is given by the  \emph{within-cluster sum of squares}. This setting gives rise to the \emph{$K$-means clustering} \cite{macqueen1967,lloyd1982least} which is one of the most popular approaches for 
	%widely regarded as the \emph{de facto} standard for cluster analysis.
	%One of the most popular approaches for cluster an
	Formally, $K$-means clustering  seeks for a solution to the mathematical optimization problem
	\begin{equation}
		\label{eq:kmeans}
		\begin{array}{cll}
			\min&\displaystyle\sum_{i=1}^K\sum_{n\in\mathcal P_i}\|\bm x_n-\bm c_i\|^2\\
			\st& \displaystyle\mathcal P_i\subseteq\{1,\dots,N\},\;\bm c_i\in\RR^D&\forall i\in\{1,\dots,K\}\\
			&\displaystyle\bm c_i=\frac{1}{|\mathcal P_i|}\sum_{n\in\mathcal P_i}\bm x_n\\
			&\mathcal P_1\cup\dots\cup\mathcal P_K=\{1,\dots,N\},\;\;\mathcal P_i\cap\mathcal P_j=\emptyset & \forall i,j\in\{1,\dots,K\}:i\neq j.
		\end{array}
	\end{equation}
	Here, $\bm x_1,\dots,\bm x_N$ are the input data points, while $\mathcal P_1,\dots,\mathcal P_K\subseteq \{1,\dots,N\}$ are the output clusters. The vectors $\bm c_1,\dots,\bm c_K\in\RR^D$ in \eqref{eq:kmeans} determine  the  cluster centroids, while the constraints on the last row of \eqref{eq:kmeans} ensure that the subsets $\mathcal P_1,\dots,\mathcal P_K$ constitute a partition of the set $\{1,\dots,N\}$. 
	
	Due to its combinatorial nature, the $K$-means clustering problem \eqref{eq:kmeans} is generically NP-hard \cite{aloise2009np}. 
	A popular solution scheme  for this intractable problem is  the heuristic algorithm developed by Lloyd \cite{lloyd1982least}. The algorithm  initializes by randomly selecting  $K$ cluster centroids. It then proceeds by alternating between the \emph{assignment} step and the \emph{update} step. In {the} assignment step the algorithm designates each data point to the closest centroid, while in {the} update step the algorithm determines  new cluster centroids according to current assignment. 
	
	Another popular {solution} approach arises in the form of convex relaxation schemes \cite{peng2007approximating,awasthi2015relax,rujeerapaiboon2017size}. In this approach, tractable  semidefinite  programming (SDP) lower bounds  for \eqref{eq:kmeans} are derived. Solutions of these optimization problems are then transformed into  cluster assignments via well-constructed rounding procedures. Such convex relaxation schemes have a number of theoretically appealing properties. If the data points are supported on $K$ disjoint balls then exact recovery is possible with high probability whenever the distance between any two balls is sufficiently large \cite{awasthi2015relax,iguchi2015tightness}. A stronger model-free result is achievable if the cardinalities of the clusters are prescribed to the problem~\cite{rujeerapaiboon2017size}.  
	
	A closely related problem is the non-negative matrix factorization with orthogonality constraints (ONMF). Given an input data matrix $\bm X$, the ONMF problem seeks for non-negative matrices $\bm F$ and $\bm U$ so that both the product $\bm F\bm U^\top$ is close to $\bm X$ in view of the Frobenius norm and the orthogonality constraint $\bm U^\top\bm U=\mathbb I$ is satisfied. Although ONMF is not precisely equivalent to $K$-means, solutions to this problem have clustering property \cite{ding2005equivalence,li2006relationships,ding2006orthogonal,kuang2012symmetric}. In  \cite{pompili2014two}, 
	it is shown that the  ONMF problem is in fact equivalent to a weighted variant of the $K$-means clustering problem. 
	
	In this paper, we attempt to obtain equivalent convex reformulations for the ONMF and  $K$-means clustering problems.  To derive these reformulations, we adapt the  results by  Burer and {Dong}~\cite{burer2012representing} who show that any (non-convex) quadratically constrained quadratic program (QCQP) can  be reformulated as a linear program over the convex cone of completely positive matrices. The resulting optimization problem is called a \emph{{generalized} completely positive program}. Such a transformation does not immediately mitigate the intractability of the original problem, since solving a generic completely positive program is  NP-hard. However, the complexity of the problem is now entirely absorbed in the   cone  of completely positive matrices which admits tractable semidefinite representable outer approximations~\cite{parrilo2000structured,DKP02:copositive,lasserre2009convexity}. Replacing the cone with  these outer approximations gives rise to SDP relaxations of the original problem  that in principle can be solved efficiently. 
	
	As byproducts of our derivations, we identify a new condition that makes the ONMF and {the} $K$-means clustering problems equivalent and we obtain new SDP relaxations for the $K$-means clustering problem that are tighter than the well-known relaxation proposed by Peng and Wei \cite{peng2007approximating}. 
	%The current solution schemes for this problem can be classified into two major classes. The first class 
	%so that all points in each cluster are close to each other with respect to a prespecified distance criterion. 
	%\newpage
	The contributions of this paper can be summarized as follows. 
	\begin{enumerate}
		\item We disclose a new connection between ONMF and $K$-means clustering. We show that $K$-means clustering is equivalent to ONMF  if an additional requirement on the binarity of  solution  to the latter problem is imposed. This  amends the previous incorrect result by Ding et al.~\cite[Section 2]{ding2005equivalence} and Li and Ding \cite[Theorem 1]{li2006relationships} who claimed that both problems are equivalent.\footnote{To our best understanding, they have shown only  one of the implications that establish an equivalence.}
		\item %The newly established connection enables us to
		We derive exact conic programming reformulations for the ONMF and $K$-means clustering problems that are principally amenable to numerical solutions. To our best knowledge, we are the first to obtain equivalent convex reformulations for these problems. 
		\item In view of {the} equivalent convex reformulation, we derive tighter SDP relaxations for the $K$-means clustering problem whose %Any rounding schemes for this improved relaxation will never generate cluster assignments that are inferior to the solution of the state-of-the-art SDP relaxation. 
		solutions can be used to construct high quality estimates of the cluster assignment. 
		\item We devise a new approximation algorithm for the $K$-means clustering problem that leverages the improved relaxation and  numerically  highlight its superiority over the state-of-the-art SDP approximation scheme by {Mixon et al.~\cite{mixon2016clustering}} and the Lloyd's algorithm. 
	\end{enumerate}
	
	The remainder of the paper is structured as follows. In Section \ref{sec:qcqp}, {we present a theorem for reformulating the QCQPs studied in the paper as generalized completely positive programs}. In Section \ref{sec:ONMF}, we derive a conic programming reformulation for the ONMF problem. We extend this result to the setting of $K$-means clustering  in Section \ref{sec:kmeans}. In Section \ref{sec:kmeans_algo}, we develop SDP relaxations  and design a new approximation algorithm for $K$-means clustering. Finally, we empirically assess the performance of our proposed algorithm in Section~\ref{sec:experiments}.

	\paragraph{Notation:} %\blue{We denote a matrix by bold capital  and a vector by bold letter}. 
	For any $K\in\mathbb N$, we define $[K]$ as the index set $\{1,\dots,K\}$. We denote by $\mathbb I$ the identity matrix and by $\mathbf e$ the vector of all ones. We also define $\mathbf e_i$ as the $i$-th canonical basis vector.  Their dimensions will be clear from the context. The trace of a square matrix $\bm M$ is denoted as $\tr(\bm M)$. We define $\diag(\bm v)$ as the diagonal matrix whose diagonal components comprise the entries of $\bm v$.  For any non-negative vector $\bm v\in\RR_+^K$, we define  the cardinality of all positive components of $\bm v$ by $\# \bm v=|\{i\in[K]:v_{i}>0\}|$.  For any matrix $\bm M\in\RR^{M\times N}$, we denote by $\bm m_i\in\RR^{M}$  the vector that corresponds to the $i$-th column of~$\bm M$. 
	The set of all symmetric matrices in $\RR^{K\times K}$ is denoted as $\mathbb S^K$, while the cone of positive semidefinite matrices in $\RR^{K\times K}$ is denoted as~$\mathbb S_+^K$. %We define the cone of copositive matrices as $\mathcal C^*=\{\bm M\in\mathbb S^K:\bm{\xi}^\top\bm{M}\bm{\xi}\geq 0\;\forall\bm{\xi}\geq \bm 0\}$ and the cone of completely positive matrices as $\mathcal C=\{\bm{M}\in\mathbb S^K:\bm M=\bm B\bm B^\top \text{ for some } \bm B\in\RR_+^{K\times g(K)}\}$, where $g(K)= \max\{{K+1\choose 2}-4,K\}$ \cite{SBBJ15:CP-RANK}.
%	\blue{Also, conv(.) denotes the convex hull, cone(.) denotes the convex conic hull, and clconv(.) denotes the closure of convex hull.}
	The cone of completely positive matrices over a set $\mathcal K$ is denoted as $\mathcal C(\mathcal K)=\cl\conv\{\bm x\bm x^\top:\bm x\in\mathcal K\}$.  
	%\{\bm{M}\in\mathbb S^K~:~\bm M=\bm B\bm B^\top \text{ for some } \bm B\in\RR_+^{K\times K}\}$. 
	For any $\bm Q, \bm R\in\mathbb S^K$ and any {closed convex} cone $\mathcal C$, the relations $\bm Q\succeq\bm R$ and $\bm Q\succeq_{\mathcal C}\bm R$ denote that $\bm Q-\bm R$ is an element of $\mathbb S_+^K$ and $\mathcal C$, respectively.  The $(K+1)$-dimensional second-order cone is defined as $\soc^{K+1}=\{(\bm x, t)\in \RR^{K+1}:\|\bm x\|\leq t\}$,  {where $\|\bm x\|$ denotes the 2-norm of the vector $\bm x$}. We denote by $\soc^{K+1}_+=\soc^{K+1}\cap\RR^{K+1}_+$ the intersection of the $K+1$-dimensional second-order cone and the non-negative orthant. %We denote the $j$-th row ($j$-th column) of a matrix $\bm M$ as $\bm M_{j:}$ ($\bm M_{:j}$). %All random variables are designated by tilde signs (\eg, $\tilde{\bm\xi}$), while their realizations are denoted without tildes (\eg, $\bm{\xi}$).  The characteristic function of a set $\mathcal S$ is defined as $\chi_{\mathcal S}(\bm \xi)=0$ if $\bm \xi\in\mathcal S$; $=\minty$ otherwise. 
	
	\section{Completely Positive Programming Reformulations of QCQPs} \label{sec:qcqp}
	To derive the equivalent completely positive programming reformulations in the subsequent sections, we first generalize the results in~\cite[Theorem 1]{burer2012representing} and {\cite[Theorem 3]{burer2012copositive}}. Consider the (nonconvex) quadratically constrained quadratic program (QCQP) given by 
	\begin{equation}
		\label{eq:QCQP}
		\begin{array}{cll}
			\min&\bm p^\top\bm C_0\bm p+2\bm c_0^\top\bm p \\
			\st& \bm p\in\mathcal K\\
			&\bm A\bm p=\bm b\\
			& \bm p^\top\bm C_j\bm p+2\bm c_j^\top\bm p=\phi_j&\forall j\in[J]
		\end{array}
	\end{equation}
	Here, $\mathcal K\subseteq \RR^D$ is a closed convex cone, while $\bm A\in\RR^{I\times D}$, $\bm b\in\RR^{I}$, $\bm C_0,\bm C_j\in\mathbb S^D$, $\bm c_0,\bm c_j\in\RR^D$, $\phi_j\in\RR$, $j\in[J]$, are the respective input problem parameters. We define the feasible set of problem \eqref{eq:QCQP} as
	\begin{equation*}
		\mathcal F=\left\{\bm p\in\mathcal K:\bm A\bm p=\bm b,\;\bm p^\top\bm C_j\bm p+2\bm c_j^\top\bm p=\phi_j\quad\forall j\in[J]\right\}
	\end{equation*}
	and the recession cone of the linear constraint system as
	$\mathcal{F}^\infty: = \{\bm d \in \mathcal{K} :  \bm A\bm d = \bm 0\}$. 
	%as the feasible set of problem \eqref{eq:QCQP}%. Let $\mathcal L=\{\bm z\in\mathcal K:\bm A\bm z=\bm b\}$ be an affine slice of a closed convex cone $\mathcal K$ 
	%		and $\mathcal{F}_\infty: = \{\bm d \in \mathcal{K} :  \bm A\bm d = \bm 0\}$.
	We further define the following subsets of $\mathcal C(\mathcal K\times \RR_+)$:
	\begin{equation}
		\label{eq:set_Q}
		\mathcal Q=\left\{\begin{bmatrix}
			\bm p\\ 1
		\end{bmatrix}\begin{bmatrix}
			\bm p\\ 1
		\end{bmatrix}^\top:\bm p\in\mathcal F\right\}\quad\text{  and  } \quad\mathcal Q^\infty=\left\{\begin{bmatrix}
			\bm d\\ 0
		\end{bmatrix}\begin{bmatrix}
			\bm d\\ 0
		\end{bmatrix}^\top:\bm d\in\mathcal F^\infty\right\}.
	\end{equation}
	A standard result in convex optimization enables us to reformulate the QCQP \eqref{eq:QCQP} as the linear convex program
	\begin{equation}
		\label{eq:QCQP_equiv}
		\begin{array}{cll}
			\min&\tr(\bm C_0\bm Q)+2\bm c_0^\top\bm p \\
			\st& \begin{bmatrix}
				\bm Q&\bm p\\
				\bm p^\top&1
			\end{bmatrix}\in\cl\conv  \left(\mathcal Q\right). 
		\end{array}
	\end{equation}
	
	Recently, Burer \cite{burer2012copositive} showed that, in the absence of quadratic constraints in $\mathcal F$, the set $\cl\conv  \left(\mathcal Q\right)$ is equal to the intersection of a polynomial size linear constraint system and a generalized completely positive cone. In~\cite{burer2012representing}, Burer and {Dong} showed that such a reformulation is achievable albeit  more cumbersome in the presence of generic quadratic constraints in $\mathcal F$.  Under some additional assumptions about the structure of the quadratic constraints, one can show that the set $\cl\conv  \left(\mathcal Q\right)$ is amenable to a much simpler completely positive reformulation (see \cite[Theorem 1]{burer2012representing} and {\cite[Theorem 3]{burer2012copositive}}).  Unfortunately, these assumptions are too restrictive to reformulate the quadratic programming instances  we study in this paper.  To that end, the following theorem provides the required extension  that will enable us to derive the equivalent completely positive programs.  
	\begin{thm}
		\label{thm:main}
		\label{GCP} 	
		%, and $\mathcal Q=\{\bm z\in\RR^K:\bm z^\top\bm C^j\bm z+2(\bm c^j)^\top\bm z=\phi^j,\;\forall j\in[J]\}$ be the intersection of level sets of $J$ quadratic functions. 
		Suppose there exists an increasing sequence of index sets $\mathcal T_0=\emptyset\subseteq\mathcal T_1\subseteq \mathcal T_2\subseteq\cdots \subseteq \mathcal T_M=[J]$ with the corresponding structured feasible sets
		\begin{equation}
			\label{eq:F_m}
			\mathcal F_m=\left\{\bm p\in\mathcal K:\bm A\bm p=\bm b,\;\bm p^\top\bm C_j\bm p+2\bm c_j^\top\bm p=\phi_j\quad\forall j\in\mathcal T_m\right\}\qquad\forall m\in[M]\cup\{0\},
		\end{equation}	
		such that for every $m\in[M]$ we have
		\begin{equation}
			\label{eq:phi_j}
			\phi_j=\displaystyle \min_{\bm p\in\mathcal F_{m-1}}\bm p^\top\bm C_j\bm p+2\bm c_j^\top\bm p\quad\textup{or}\quad{ \phi_j=\displaystyle \max_{\bm p\in\mathcal F_{m-1}}\bm p^\top\bm C_j\bm p+2\bm c_j^\top\bm p}\qquad\forall j\in\mathcal T_{m}{\setminus\mathcal T_{m-1}},
		\end{equation}
		and there exists a vector $\overline{\bm p} \in\mathcal F$ such that 
		\begin{equation}
			\label{eq:p_bar}
			\bm d^\top\bm C_j\bm d+2\bm d^\top(\bm C_j \overline{\bm p} + \bm c_j) = 0\quad ~ \forall \bm d \in \mathcal{F}^\infty \;\forall j\in[J].
		\end{equation}
		%			\begin{equation}
		%			\label{eq:phi_j}
		%			\left.\begin{array}{rcll}
		%			\phi^j=&\min&\bm z^\top\bm C_j\bm z+2\bm c_j^\top\bm z \\
		%			&\st& \bm z\in\mathcal K\\
		%			&&\bm A\bm z=\bm b\\
		%						&& \bm z^\top\bm C_k\bm z+2\bm c_k^\top\bm z=\phi_k&\forall k\in\mathcal T_{m-1} 
		%			\end{array}\right\}\qquad\forall j\in\mathcal T_m. 
		%			\end{equation}
		Then, $\cl\conv  \left(\mathcal Q\right)$ coincides with
		\begin{equation}
			\label{eq:convexified_set}
			\mathcal R=\left\{\begin{bmatrix}
				\bm Q&\bm p\\
				\bm p^\top&1
			\end{bmatrix} \in \mathcal{C}(\mathcal{K} \times \RR_+): \begin{array}{l}\displaystyle
				\bm A\bm p = \bm b, \; \diag(\bm A\bm Q\bm A^\top) = \bm b\circ\bm b\\
				\displaystyle\tr(\bm C_j\bm Q)+2\bm c_j^\top\bm p=\phi_j\quad\forall j\in[J]
			\end{array} \right\}.
		\end{equation}
	\end{thm}
	%A few remarks about Theorem \ref{thm:main} are in order. 
	%	\red{It can be shown that the theorem also holds for $\phi_j=\displaystyle \max_{\bm p\in\mathcal F_{m-1}}\bm p^\top\bm C_j\bm p+2\bm c_j^\top\bm p\;\forall j\in\mathcal T_{m}\red{\setminus\mathcal T_{m-1}}$, as it is same as $\tr(-\bm C_j\bm Q)-2\bm c_j^\top\bm p=-\phi_j$, where $-\phi_j$ is a maximum.} 
	Theorem \ref{thm:main} constitutes a generalization of the combined results of \cite[Theorem 1]{burer2012representing} and {\cite[Theorem 3]{burer2012copositive}}, which we state in the following proposition. 
	\begin{prop}
		\label{prop:BurerDong}
		%Let the sets  $\mathcal L$ and $\mathcal L_\infty$ be defined as in Theorem \ref{thm:main}. 
		Let $\mathcal L=\{\bm p\in\mathcal K:\bm A\bm p=\bm b\}$. 
		Suppose $\phi_j=\min_{\bm p\in\mathcal L}\bm p^\top\bm C_j\bm p+2\bm c_j^\top\bm p$, {and} both $\min_{\bm p\in\mathcal L}\bm p^\top\bm C_j\bm p+2\bm c_j^\top\bm p$ and $\max_{\bm p\in\mathcal L}\bm p^\top\bm C_j\bm p+2\bm c_j^\top\bm p$ {are} finite for all $j\in[J]$. If there exists $\overline{\bm p} \in \mathcal F$ such that $\bm d^\top(\bm C_j \overline{\bm p} + \bm c_j) = 0$ for all $\bm d \in \mathcal{F}^\infty$ and $j\in[J]$,
		%If $\bm d^\top\bm C^j\bm d=0$ for all $\bm d\in\mathcal L_\infty$ 
		then  $\cl\conv  \left(\mathcal Q\right)$  coincides with~$\mathcal R$. 
		%		For a given quadratic equation $\bm z^\top \bm M \bm z + 2 (\bm m)^\top \bm z = \gamma$, consider a symmetric matrix variable $\bm Z = \bm z \bm z^\top$ and a nonempty set of the form $\mathcal{L} \cap \mathcal{Q}$, where $\mathcal{L}: = \{\bm z \in \mathcal{K} :  \bm A\bm z = \bm b\}$ is an affine slice of a closed convex cone $\mathcal{K}$, $\mathcal{L}_\infty: = \{\bm d \in \mathcal{K} :  \bm A\bm d = 0\}$ is the recession cone, and $\mathcal{Q} : = \{\bm z \ |\ \bm z^\top \bm M \bm z + 2 (\bm m)^\top \bm z = (\phi)_* \}$ is the level set of quadratic function, where
		%		\begin{equation*}
		%		(\phi)_* := \min_{\bm z \in \mathcal{L}}{(\bm z^\top \bm M \bm z + 2 (\bm m)^\top \bm z)}, \ \ (\phi)^* := \max_{\bm z \in \mathcal{L}}{(\bm z^\top \bm M \bm z + 2 (\bm m)^\top \bm z)}. 
		%		\end{equation*} \\
		%		Suppose both $(\phi)_* \text{ and } (\phi)^*$ are finite and there exists $\overline{\bm z} \in \mathcal{L}\cap\mathcal{Q}$ such that $\bm d^\top(\bm M \overline{\bm z} + \bm m) = 0, \forall \bm d \in \mathcal{L}_\infty$. Then the feasible region $\mathcal{L} \cap \mathcal{Q}$ can be convexified $(\mathcal{C^\prime}(\mathcal{L}\cap\mathcal{Q}))$ via 
		%		\begin{align*}
		%		\mathcal{C^\prime}(\mathcal{L} \cap \mathcal{Q}) = 
		%		&\bigg\{(\bm z, \bm Z) | \begin{bmatrix}
		%		\bm Z&\bm z\\
		%		\bm z^\top&1
		%		\end{bmatrix} \in \mathcal{C}(\mathcal{K} \times \RR_+), 
		%		\\ 	&\bm A\bm z = \bm b, \ \ (\bm A\bm Z\bm A^\top)_{ii} = \bm b^2_i \ \ \forall i, 
		%		\\ &\tr(\bm M\bm Z) + 2 (\bm m)^\top \bm z = (\phi)^* \bigg\}.
		%		\end{align*}			
	\end{prop}
\noindent	To see this, assume that all conditions in Proposition \ref{prop:BurerDong} are satisfied. Then, setting $M=1$ and $\mathcal T_1=[J]$, we find that the condition~\eqref{eq:phi_j} in Theorem \ref{thm:main} is satisfied. Next, for every $j\in[J]$, the finiteness of both $\min_{\bm p\in\mathcal L}\bm p^\top\bm C_j\bm p+2\bm c_j^\top\bm p$ and $\max_{\bm p\in\mathcal L}\bm p^\top\bm C_j\bm p+2\bm c_j^\top\bm p$ implies that $\bm d^\top\bm C_j\bm d=0$ for all $\bm d\in\mathcal F^\infty$. Combining this with the last condition in Proposition \ref{prop:BurerDong}, we find that there there exists a vector $\overline{\bm p} \in \mathcal F$ such that $\bm d^\top\bm C_j\bm d+2\bm d^\top(\bm C_j \overline{\bm p} + \bm c_j) = 0$ for all $\bm d \in \mathcal{F}^\infty$ and $j\in[J]$. Thus, all conditions in Theorem \ref{thm:main} are indeed satisfied.
	
	%		[Explain why we need this generalization]
	
	In the remainder of the section, we define the sets
	%			\begin{equation}
	%			\label{eq:F_m}
	%			\mathcal F_m=\left\{\bm z\in\mathcal K:\bm A\bm z=\bm b,\;\bm z^\top\bm C_j\bm z+2\bm c_j^\top\bm z=\phi_j\quad\forall j\in\mathcal T_m\right\},
	%			\end{equation}
	\begin{equation*}
		\mathcal Q_m	
		=\left\{\begin{bmatrix}
			\bm p\\ 1
		\end{bmatrix}\begin{bmatrix}
			\bm p\\ 1
		\end{bmatrix}^\top:\bm p\in\mathcal F_m\right\}
		\;\textup{ and }\;			\mathcal R_m=\left\{\begin{bmatrix}
			\bm Q&\bm p\\
			\bm p^\top&1
		\end{bmatrix} \in \mathcal{C}(\mathcal{K} \times \RR_+): \begin{array}{l}\displaystyle
			\bm A\bm p = \bm b\\\diag(\bm A\bm Q\bm A^\top) = \bm b\circ\bm b\\
			\displaystyle\tr(\bm C_j\bm Q)+2\bm c_j^\top\bm p=\phi_j\;\;\forall j\in\mathcal T_m
		\end{array} \right\}
	\end{equation*}
	for $m\in[M]\cup\{0\}$. 
	The proof of Theorem \ref{thm:main} relies on the following lemma, which is established in the first part of the proof of {\cite[Theorem 3]{burer2012copositive}}. 
	\begin{lem}
		\label{lem:conv_cone_clconv}
		Suppose there exists a vector $\overline{\bm p} \in\mathcal F$ such that $\bm d^\top\bm C_j\bm d+2\bm d^\top(\bm C_j \overline{\bm p} + \bm c_j) = 0$ for all $\bm d \in \mathcal{F}^\infty$ and $j\in[J]$, then we have
		$\conv(\mathcal Q_m)+\cone(\mathcal Q^\infty) \subseteq \cl\conv(\mathcal Q_m)$ for all $m\in[M]$. 
	\end{lem}
	
	\noindent Using this lemma, we are now  ready to prove Theorem \ref{thm:main}. 
	
	%		 is immediate from the proofs of \cite[Theorem 1]{burer2012representing} and {\cite[Theorem 3]{burer2012copositive}}. % We only highlight the major differences as follows. 
	
	\begin{proof}[Proof of Theorem \ref{thm:main}]
		%			For every $m\in[M]$, we define the sets 
		%			\begin{equation*}
		%			\mathcal F_m=\left\{\bm z\in\mathcal K:\bm A\bm z=\bm b,\;\bm z^\top\bm C^j\bm z+2(\bm c^j)^\top\bm z=\phi^j\quad\forall j\in\mathcal T_m\right\}
		%			\end{equation*}
		%			and
		%			\begin{equation*}
		%			\mathcal R_m=\left\{(\bm z,\bm Z)\in\RR^D\times \mathbb S^D: \begin{array}{l}\displaystyle\begin{bmatrix}
		%			\bm Z&\bm z\\
		%			\bm z^\top&1
		%			\end{bmatrix} \in \mathcal{C}(\mathcal{K} \times \RR_+),\;
		%			\bm A\bm z = \bm b, \; \diag(\bm A\bm Z\bm A^\top) = \bm b\circ\bm b\\
		%			\displaystyle\tr(\bm C^j\bm Z)+2(\bm c^j)^\top\bm z=\phi^j\quad\forall j\in\mathcal T_m
		%			\end{array} \right\}.
		%			\end{equation*}
		The proof  follows if $\cl\conv(\mathcal Q_m)=\mathcal R_m$ for all $m\in[M]$. By construction, we have $\cl\conv(\mathcal Q_{m})\subseteq \mathcal R_{m}$, $m\in[M]$.  	 It thus remains to prove the converse inclusions.  By Lemma \ref{lem:conv_cone_clconv}, it suffices to show that $\mathcal R_{m}\subseteq\conv(\mathcal Q_{m})+\cone(\mathcal Q^\infty)$ for all $m\in[M]$. We proceed via induction. The base case for $m=0$ follows from~\cite[Theorem 1]{burer2012copositive}. Assume now that $\mathcal R_{m-1}\subseteq\conv(\mathcal Q_{m-1})+\cone(\mathcal Q^\infty)$ holds for a positive index $m-1<M$. We will show that  this implies $\mathcal R_{m}\subseteq\conv(\mathcal Q_{m})+\cone(\mathcal Q^\infty)$. To this end, consider the following completely positive decomposition of an element of $\mathcal R_{m}$:
		\begin{equation}
			\label{eq:CPP_decomposition}
			\begin{bmatrix}
				\bm Q&\bm p\\
				\bm p^\top&1
			\end{bmatrix} = \sum_{s\in\mathcal S} \begin{bmatrix}
				\bm \zeta_s\\
				\eta_s
			\end{bmatrix}\begin{bmatrix}
				\bm \zeta_s\\
				\eta_s
			\end{bmatrix}^\top=\sum_{s\in\mathcal S_+} \eta_s^2\begin{bmatrix}
				\bm \zeta_s/\eta_s\\
				1
			\end{bmatrix}\begin{bmatrix}
				\bm \zeta_s/\eta_s\\
				1
			\end{bmatrix}^\top+
			\sum_{s\in\mathcal S_0} \begin{bmatrix}
				\bm \zeta_s\\
				0
			\end{bmatrix}\begin{bmatrix}
				\bm \zeta_s\\
				0
			\end{bmatrix}^\top.
		\end{equation}
		Here, $\mathcal S_+=\{s\in\mathcal S:\eta_s>0\}$ and $\mathcal S_0=\{s\in\mathcal S:\eta_s=0\}$, where $\mathcal S$ is  a finite index set. By our induction hypothesis, we have $\bm \zeta_s/\eta_s\in\mathcal F_{m-1}$, $s\in\mathcal S_+$, and $\bm{\zeta}_s\in\mathcal F^\infty$, $s\in\mathcal S_0$. The proof thus follows if the constraints 
		\begin{equation*}
			\tr(\bm C_j\bm Q)+2\bm c_j^\top\bm p=\phi_j\quad\forall j\in\mathcal T_{m}\setminus\mathcal T_{m-1}
		\end{equation*}
		in $\mathcal R_{m}$ imply
		\begin{equation*}
			(\bm \zeta_s/\eta_s)^\top\bm C_j(\bm \zeta_s/\eta_s)+2\bm c_j^\top(\bm \zeta_s/\eta_s)=\phi_j\quad\forall j\in\mathcal T_{m}\setminus\mathcal T_{m-1}.
		\end{equation*}
		Indeed, for every {$j\in\mathcal T_{m}\setminus\mathcal T_{m-1}$}, the decomposition  \eqref{eq:CPP_decomposition} yields
		\begin{align*}
			\phi_j&\displaystyle=\tr(\bm C_j\bm Q)+2\bm c_j^\top\bm p\\
			&\displaystyle=\sum_{s\in\mathcal S_+}\eta_s^2\left[(\bm \zeta_s/\eta_s)^\top\bm C_j(\bm \zeta_s/\eta_s)+2\bm c_j^\top(\bm \zeta_s/\eta_s)\right]+\sum_{s\in\mathcal S_0}\bm \zeta_s^\top\bm C_j\bm \zeta_s\\
			&\displaystyle=\sum_{s\in\mathcal S_+}\eta_s^2\left[(\bm \zeta_s/\eta_s)^\top\bm C_j(\bm \zeta_s/\eta_s)+2\bm c_j^\top(\bm \zeta_s/\eta_s)\right].
		\end{align*}
		Here, the last equality follows from our assumption that there exists a vector $\overline{\bm p} \in\mathcal F$ such that $\bm d^\top\bm C_j\bm d+2\bm d^\top(\bm C_j \overline{\bm p} + \bm c_j) = 0$ for all $\bm d\in\mathcal F^\infty$. Thus, $\bm d^\top\bm C_j\bm d=0$ for all $\bm d\in\mathcal F^\infty$. Next, since $\bm \zeta_s/\eta_s\in\mathcal F_{m-1}$, the $j$-th identity in \eqref{eq:phi_j} implies that {$(\bm \zeta_s/\eta_s)^\top\bm C_j(\bm \zeta_s/\eta_s)+2\bm c_j^\top(\bm \zeta_s/\eta_s)\geq \phi_j$ if $\phi_j=\displaystyle \min_{\bm p\in\mathcal F_{m-1}}\bm p^\top\bm C_j\bm p+2\bm c_j^\top\bm p$ or  $(\bm \zeta_s/\eta_s)^\top\bm C_j(\bm \zeta_s/\eta_s)+2\bm c_j^\top(\bm \zeta_s/\eta_s)\leq \phi_j$ if $\phi_j=\displaystyle \max_{\bm p\in\mathcal F_{m-1}}\bm p^\top\bm C_j\bm p+2\bm c_j^\top\bm p$}. The proof thus follows since $\eta_s^2>0$ and $\sum_{s\in\mathcal S_+}\eta_s^2=1$.
	\end{proof}

	\section{Orthogonal Non-Negative Matrix Factorization}
	\label{sec:ONMF}
	In this section, we first consider the ONMF problem given by
	\begin{equation}
		\label{eq:NMF}
		\begin{array}{cl}
			\min&\|\bm X-\bm H\bm U^\top\|_F^2\\
			\st& \bm H\in\RR_+^{D\times K},\;\bm U\in\RR_+^{N\times K}\\
			&\bm U^\top\bm U=\mathbb I. 
		\end{array}
	\end{equation}
	Here, $\bm X\in\RR^{D\times N}$ is a matrix whose columns comprise $N$ data points $\{\bm x_n\}_{n\in[N]}$ in $\RR^D$. We remark that  problem \eqref{eq:NMF} is generically intractable since we are minimizing a non-convex quadratic objective function over the Stiefel manifold~\cite{absil2009optimization,asteris2014nonnegative}. 
	%In this following, we aim to derive a convex reformulation for the generic problem \eqref{eq:NMF}. To this end, 
	By expanding the Frobenius norm in the objective function and noting that $\bm U^\top\bm U=\mathbb I$, we find that problem \eqref{eq:NMF} is equivalent to 
	\begin{equation}
		\label{eq:NMF_}
		\begin{array}{cl}
			\min&\tr\left(\bm X^\top\bm X-2\bm X\bm U \bm H^\top+\bm H^\top\bm H\right)\\
			\st& \bm H\in\RR_+^{D\times K},\;\bm U\in\RR_+^{N\times K}\\
			&\bm U^\top\bm U=\mathbb I. 
		\end{array}
	\end{equation}
	We now derive a convex reformulation for  problem \eqref{eq:NMF_}. We remark that this problem is still intractable due to non-convexity of {the} objective function and {the} constraint system. Thus, any resulting convex formulation will in general remain intractable. 
In the following, to reduce the clutter in our notation, we define the convex set
	\begin{equation*}
		\mathcal W(\mathcal B,K)=\left\{\left((\bm p_i)_{i\in[K]},(\bm Q_{ij})_{i,j\in[K]}\right):\begin{bmatrix}
			\bm Q_{11}&\cdots & \bm Q_{1K}&\bm p_1\\
			\vdots &\ddots&\vdots&\vdots\\
			\bm Q_{K1}&\cdots&\bm Q_{KK}&\bm p_K\\
			\bm p_{1}^\top&\cdots&\bm p_{K}^\top&1\\
		\end{bmatrix}\in\mathcal C\left(\mathcal B^K\times \RR_+\right)\right\},
	\end{equation*}
	{\noindent 	where $\bm p_i \in \mathcal B$ and $\bm Q_{ij} \in \RR^{(N+1+D) \times (N+1+D)}_+$, $i, j \in [K]$. Here, $\mathcal B$ is a given convex cone, $K$ is a positive integer, and $\mathcal B^K$ is the direct product of $K$ copies of $\mathcal B$. }
	\begin{thm}
		\label{thm:CPP}
		Problem \eqref{eq:NMF_} is equivalent to the following {generalized} completely positive program:
		\begin{equation}
			\label{eq:ONMF_CPP}
			\begin{array}{cll}
				\min&\displaystyle\tr(\bm X^\top\bm X)+\sum_{i\in[K]}\tr(-2\bm X\bm W_{ii}+\bm G_{ii})\\
				\st& \left((\bm p_i)_{i\in[K]},(\bm Q_{ij})_{i,j\in[K]}\right)\in	\mathcal W\left(\soc_+^{N+1}\times\RR_+^D,K\right)\\
				& \bm u_i\in\RR_+^N,\;\bm V_{ij}\in \RR_+^{N\times N},\;\bm h_i\in\RR_+^D,\;\bm G_{ij}\in\RR_+^{D\times D},\;\bm W_{ij}\in\RR_+^{N\times D}&\forall i,j\in[K]\\
				&\bm p_i=\begin{bmatrix}
					\bm u_i\\ 1\\ \bm h_i
				\end{bmatrix},\;\bm Q_{ij}=\begin{bmatrix}\bm V_{ij}&\bm u_i&\bm W_{ij}\\ \bm u_j^\top & 1 &\bm h_j^\top\\ \bm W_{ji}^\top& \bm h_i&\bm G_{ij}\end{bmatrix}&\forall i,j\in[K]\\ 
				&\tr(\bm V_{ii})=1 &\forall i\in[K]\\ 
				&\tr(\bm V_{ij})=0 &\forall i,j\in[K]:i\neq j.
				%		&\begin{bmatrix}
				%		\bm Q_{11}&\cdots & \bm Q_{1K}&\bm p_1\\
				%		\vdots &\ddots&\vdots&\vdots\\
				%		\bm Q_{K1}&\cdots&\bm Q_{KK}&\bm p_K\\
				%		\bm p_{1}^\top&\cdots&\bm p_{K}^\top&1\\
				%		\end{bmatrix}\in\mathcal C\left((\soc_+^{N+1}\times\RR_+^D)^K\times \RR_+\right)
			\end{array}
		\end{equation}			
	\end{thm}

	\begin{proof}
		By utilizing the notation for column vectors $\{\bm u_i\}_{i\in[K]}$ and $\{\bm h_i\}_{i\in[K]}$, we can reformulate  problem~\eqref{eq:NMF_} equivalently as the problem
		\begin{equation}
			\label{eq:NMF_1}
			\begin{array}{cll}
				\min&\displaystyle\tr(\bm X^\top\bm X)-2\sum_{i\in[K]}\tr(\bm X\bm u_i\bm h_i^\top)+\sum_{i\in[K]}\tr(\bm h_i\bm h_i^\top)\\
				\st& \bm h_i \in\RR_+^{D},\;\bm u_i \in\RR_+^{N}&\forall i\in[K]\\
				&\bm u_i^\top\bm u_i=1&\forall i\in[K]\\
				&\bm u_i^\top\bm u_j=0&\forall i,j\in[K]:i\neq j.
			\end{array}
		\end{equation}	
		We now employ Theorem \ref{thm:main} to show the equivalence of problems \eqref{eq:NMF_1} and \eqref{eq:ONMF_CPP}. 
		We first introduce  an auxiliary decision variable $\bm p=(\bm p_1,\ldots,\bm p_K)$ that satisfies
		\begin{equation*}
			\bm p_i=\begin{bmatrix}
				\bm u_i\\ t_i \\ \bm h_i
			\end{bmatrix}\in\soc_+^{N+1}\times\RR^{D}_+\qquad\forall i\in[K].
		\end{equation*}
		%			Since $\bm u_i^\top\bm u_i=1$, and consequently $\|\bm u_i\|^2_2= 1$, we can without loss of generality impose the vector $\bm{p}_i$ to reside in~$\soc_+^{N+1}\times\RR^{D}_+$. 
		Let $M=1$ in Theorem \ref{thm:main} and set $\mathcal K=(\soc_+^{N+1}\times\RR^{D}_+)^K$. %
		We then define the structured feasible sets
		\begin{equation*}
			\mathcal F_0=\left\{\bm p\in\mathcal K:t_i=1\quad\forall i\in[K]\right\}\quad\textup{and}\quad \mathcal F_1=\mathcal  F=\left\{
			\bm p\in\mathcal F_0:\begin{array}{ll}
				\bm u_i^\top\bm u_i=1&\forall i\in[K]\\
				\bm u_i^\top\bm u_j=0&\forall i,j\in[K]:i\neq j
			\end{array}\right\}.
		\end{equation*}
		Note that for every $i\in[K]$, the constraints $\|\bm u_i\|_2\leq t_i$ and $t_i=1$ in $\mathcal F_0$ imply that the variables $\bm u_i$ and $t_i$ are bounded. Thus, the recession cone of $\mathcal F_0$ coincides with the set $
		\mathcal F^\infty=\left\{\bm p\in\mathcal K:\bm u_i=\bm 0,\;t_i=0\;\forall i\in[K]\right\}$. Next, we set the vector $\overline{\bm p}=(\overline{\bm p}_1,\ldots,\overline{\bm p}_K)\in\mathcal F$  in Theorem \ref{thm:main} to satisfy
		\begin{equation*}
			\overline{\bm p}_i=\begin{bmatrix}
				\overline{\bm u}_i\\ 1\\ \bm 0
			\end{bmatrix}\in \soc_+^{N+1}\times\RR^{D}_+\qquad\forall i\in[K],
		\end{equation*}
		where the subvectors $\{\overline{\bm u}_i\}_{i\in[K]}$ are chosen to be feasible in \eqref{eq:NMF_1}.
		In view of the description of recession cone~$\mathcal F^\infty$ and the structure of quadratic constraints in $\mathcal F$, one can readily verify that such a vector $\overline{\bm p}$ satisfies the condition \eqref{eq:p_bar} in Theorem \ref{thm:main}. %, for all \blue{$d = [\bm 0 \;\ 0 \;\ \bm h_i]^\top \in \mathcal{F}^\infty$}. 
		It remains to show that  condition \eqref{eq:phi_j} is also satisfied. %By multiplying both sides of the equality $\bm u_i^\top\bm u_i=1$ with $-1$, we find that the 
		%constraint  is equivalent to $-\bm u_i^\top\bm u_i=-1$. 
		%			\begin{equation*}
		%			\bm u_i^\top\bm u_i=1\Longleftrightarrow-\bm u_i^\top\bm u_i=-1. 
		%			\end{equation*} 
		%			Thus, 
		Indeed, we have
		\begin{equation*}
			\max_{\bm p\in\mathcal F_0}\left\{\bm u_i^\top\bm u_i\right\}=1\quad\forall i\in[K],
		\end{equation*}
		since the constraints $\|\bm u_i\|_2\leq 1$, $i\in[K]$, are implied by $\mathcal F_0$, while equalities are attained whenever the 2-norm of each vector $\bm u_i$ is 1. Similarly,  we find that
		\begin{equation*}
			\min_{\bm p\in\mathcal F_0}\left\{\bm u_i^\top\bm u_j\right\}=0\quad\forall i,j\in[K]:i\neq j,
		\end{equation*}
		since the constraints $\bm u_i\geq \bm 0$, $i\in[K]$, are implied by $\mathcal F_0$, while equalities are attained whenever the solutions $\bm u_i$ and $\bm u_j$ satisfy the complementarity property:
		\begin{equation*}
			u_{in}>0\implies u_{jn}=0\;\textup { and }\;			u_{jn}>0\implies u_{in}=0\qquad \forall n\in[N]. 
		\end{equation*}
		Thus, all conditions in Theorem \ref{thm:main} are satisfied. 
		
		Next, we introduce new matrix variables that represent a linearization of the quadratic variables, as follows:
		\begin{align}
			\label{matvar}
			\bm V_{ij} = \bm u_i \bm u_j^\top, \bm W_{ij} = \bm u_i \bm h_j^\top,\; \text{ and } \;\bm G_{ij} = \bm h_i \bm h_j^\top\qquad \forall i, j \in [K].		
		\end{align}
		We also define an auxiliary decision variable $\bm Q=(\bm Q_{ij})_{i,j\in[K]}$ satisfying
		\begin{align*}
			\bm Q_{ij} = \bm p_i \bm p_j^\top = \begin{bmatrix}\bm V_{ij}&\bm u_i&\bm W_{ij}\\ \bm u_j^\top & 1 &\bm h_j^\top\\ \bm W_{ji}^\top& \bm h_i&\bm G_{ij}\end{bmatrix}\qquad\forall i,j\in[K]. 
		\end{align*}
		Using these new terms, we construct the set $\mathcal R$ in Theorem \ref{thm:main} as follows:
		\begin{equation*}
			\mathcal R=\left\{\begin{bmatrix}
				\bm Q_{11}&\cdots & \bm Q_{1K}&\bm p_1\\
				\vdots &\ddots&\vdots&\vdots\\
				\bm Q_{K1}&\cdots&\bm Q_{KK}&\bm p_K\\
				\bm p_{1}^\top&\cdots&\bm p_{K}^\top&1\\
			\end{bmatrix} \in\mathcal{C}(\mathcal K \times \RR_+):
			\begin{array}{ll}
				\bm p_i=\begin{bmatrix}
					\bm u_i\\ 1 \\ \bm h_i
				\end{bmatrix} ,\;\bm Q_{ij} =\begin{bmatrix}\bm V_{ij}&\bm u_i&\bm W_{ij}\\ \bm u_j^\top & 1 &\bm h_j^\top\\ \bm W_{ji}^\top& \bm h_i&\bm G_{ij}\end{bmatrix}
				%\soc_+^{N+1}\times\RR^{D}_+
				&\forall i,j\in[K]\\
				%\bm Q_{ij} =\begin{bmatrix}\bm V_{ij}&\bm u_i&\bm W_{ij}\\ \bm u_j^\top & 1 &\bm h_j^\top\\ \bm W_{ji}^\top& \bm h_i&\bm G_{ij}\end{bmatrix}&\forall i,j\in[K]\\
				\tr(\bm V_{ii}) = 1& \forall i \in [K] \\ 
				\tr(\bm V_{ij}) = 0 & \forall i, j \in [K]: i \neq j	
			\end{array}\right\}. 
		\end{equation*}
		By Theorem \ref{thm:main}, this set coincides with $\cl\conv  \left(\mathcal Q\right)$, where the set $\mathcal Q$ is defined as in \eqref{eq:set_Q}. 
		%			\begin{equation*}
		%			\mathcal Q	
		%			=\left\{\begin{bmatrix}
		%			\bm p\\ 1
		%			\end{bmatrix}\begin{bmatrix}
		%			\bm p\\ 1
		%			\end{bmatrix}^\top:\bm p\in\mathcal F\right\}.
		%			\end{equation*} 
		Thus, by linearizing the objective function using the matrix variables in \eqref{matvar}, we  find that  the generalized completely positive program \eqref{eq:ONMF_CPP} is indeed equivalent to \eqref{eq:NMF_}.		This completes the proof. 
	\end{proof}
	Let us now consider a special case of problem \eqref{eq:NMF}; if all components of $\bm X$ are non-negative, then we can reduce the problem into a simpler one involving only the decision matrix $\bm U$.
	\begin{lem}
		If $\bm X$ is a non-negative matrix then problem \eqref{eq:NMF} is equivalent to the non-convex program
		\begin{equation}
			\label{eq:NMF1}
			\begin{array}{cl}
				\min&\tr(\bm X^\top\bm X-\bm X^\top\bm X\bm U\bm U^\top)\\
				\st& \bm U\in\RR_+^{N\times K}\\
				&\bm U^\top\bm U=\mathbb I. 
			\end{array}
		\end{equation}	
	\end{lem}
	\begin{proof}
		Solving the minimization over ${\bm H}\in\RR_+^{D\times K}$ analytically in \eqref{eq:NMF_}, we find that the solution ${\bm H}=\bm X\bm U$ is feasible and  optimal. Substituting this solution into the objective function of \eqref{eq:NMF_}, we arrive at the equivalent problem~\eqref{eq:NMF1}. This completes the proof. 
	\end{proof}
	\noindent By employing the same reformulation techniques as in the proof of Theorem \ref{thm:CPP}, we can show that  problem~\eqref{eq:NMF1} is amenable to an exact convex reformulation. 
	\begin{prop}
		\label{thm:CPP_1}
		Problem \eqref{eq:NMF1} is equivalent to the following {generalized} completely positive program:
		\begin{equation}
			\label{eq:NMF_CPP}
			\begin{array}{cll}
				\min&\displaystyle\tr(\bm X^\top\bm X)-\sum_{i\in[K]}\tr(\bm X^\top\bm X\bm V_{ii})\\
				\st& \left((\bm p_i)_{i\in[K]},(\bm Q_{ij})_{i,j\in[K]}\right)\in	\mathcal W\left(\soc_+^{N+1},K\right),\;
				%\bm p_i\in\soc_+^{N+1},\;\bm Q_{ij}\in \RR_+^{(N+1)\times (N+1)},\; 
				\bm u_i\in\RR_+^N%,\;\blue{\bm (V_{ij})_{i, j \in [K]} \in \mathcal{V}(N, K)} &\forall i,j\in[K]
				\\
				&\bm p_i=\begin{bmatrix}
					\bm u_i\\ 1
				\end{bmatrix},\;\bm Q_{ij}=\begin{bmatrix}\bm V_{ij}&\bm u_i\\ \bm u_j^\top & 1 \end{bmatrix}&\forall i,j\in[K]\\ 
				&\tr(\bm V_{ii})=1 &\forall i\in[K]\\ 
				&\tr(\bm V_{ij})=0 &\forall i,j\in[K]:i\neq j.
				%		&\begin{bmatrix}
				%		\bm Q_{11}&\cdots & \bm Q_{1K}&\bm p_1\\
				%		\vdots &\ddots&\vdots&\vdots\\
				%		\bm Q_{K1}&\cdots&\bm Q_{KK}&\bm p_K\\
				%		\bm p_{1}^\top&\cdots&\bm p_{K}^\top&1\\
				%		\end{bmatrix}\in\mathcal C\left((\soc_+^{N+1})^K\times \RR_+\right).
				%&\blue{\mathcal{W}(K) \in\mathcal C\left((\soc_+^{N+1})^K\times \RR_+\right)}
			\end{array}
		\end{equation}	
	\end{prop}
	\section{$K$-means Clustering}
	\label{sec:kmeans}
	Building upon the results {from the} previous sections, we now derive an exact generalized completely positive programming reformulation for the $K$-means clustering problem  \eqref{eq:kmeans}.  %The objective is this problem is to partition the data points $\{\bm x_n\}_{n\in[N]}$ into $K$ disjoint clusters so that all  points in each cluster are close to the centroid of the cluster with respect to a norm-squared distance criterion. 
	To this end, we note that the  problem  can equivalently be solved via the following mixed-integer nonlinear program \cite{hansen1997cluster}:
	\begin{equation}
		\label{eq:kmeans_0}
		\begin{array}{ccll}
			Z^\star=&\displaystyle\min&\displaystyle\sum_{i\in[K]}\sum_{n:\pi_{in}=1}\|\bm x_n-\bm c_i\|^2\\
			&\displaystyle\st&\displaystyle\bm \pi_i\in\{0,1\}^N,\;\bm c_i\in\RR^D&\forall i\in[K]\\
			&&\displaystyle\bm c_i=\frac{1}{\mathbf e^\top\bm\pi_i}\sum_{n:\pi_{in}=1}\bm x_n&\forall i\in[K]\\
			&&\displaystyle\mathbf e^\top\bm\pi_i\geq 1&\forall i\in[K]\\
			&&\displaystyle\sum_{i\in[K]} \bm{\pi}_i= \mathbf e.
		\end{array}
	\end{equation}		
	Here, $\bm c_i$ is the centroid of the $i$-th cluster, while $\bm \pi_i$ is the assignment vector for the $i$-th cluster, \ie, $\pi_{in}=1$ if and only if  the data point $\bm x_{n}$ is assigned to the cluster $i$. The last constraint  in \eqref{eq:kmeans_0} ensures that each data point is assigned to a cluster, while the  constraint system in the penultimate row ensures that there are exactly $K$ clusters. We now show that we can solve the  $K$-means clustering problem by solving a modified problem~\eqref{eq:NMF1} with an additional constraint $\sum_{i\in[K]}\bm u_i\bm u_i^\top\mathbf e=\mathbf e$. {To further simplify our notation we will employ the sets 
		\begin{equation*}
			\mathcal U(N,K)=\left\{\bm U\in\RR_+^{N\times K}:\begin{array}{cllll}
				\bm u_i^\top\bm u_i=1&\forall i\in[K],&
				\bm u_i^\top\bm u_j=0&\forall i,j\in[K]:i\neq j
			\end{array}\right\}\;\text{ and }
		\end{equation*}
		\begin{equation*}
			\mathcal V(N,K)=\left\{(\bm V_{ij})_{i,j\in[K]}\in\RR_+^{N^2\times K^2}:\begin{array}{cllll}
				\tr(\bm V_{ii})=1 &\forall i\in[K],&
				\tr(\bm V_{ij})=0 &\forall i,j\in[K]:i\neq j
			\end{array}\right\}
		\end{equation*}
		%	\begin{align*}
		%	\blue{
		%		\mathcal W(K)} = &\begin{bmatrix}
		%	\bm Q_{11}&\cdots & \bm Q_{1K}&\bm p_1\\
		%	\vdots &\ddots&\vdots&\vdots\\
		%	\bm Q_{K1}&\cdots&\bm Q_{KK}&\bm p_K\\
		%	\bm p_{1}^\top&\cdots&\bm p_{K}^\top&1\\
		%	\end{bmatrix}
		%	\end{align*}
		in all reformulations in the remainder of this section. }
	\begin{thm}
		\label{thm:kmeans_1}
		The following non-convex program solves the $K$-means clustering problem:
		\begin{equation}
			\tag{$\mathcal Z$}
			\label{eq:kmeans_1}
			\begin{array}{ccll}
				Z^\star=&\displaystyle\min&\displaystyle\tr(\bm X^\top\bm X)-\sum_{i\in[K]}\tr(\bm X^\top\bm X\bm u_i\bm u_i^\top)\\
				&\st& {\bm U\in\mathcal{U}(N, K)}\\
				%&&\bm u_i^\top\bm u_i=1&\forall i\in[K]\\
				%&&\bm u_i^\top\bm u_j=0&\forall i,j\in[K]:i\neq j\\
				&&\displaystyle\sum_{i\in[K]}\bm u_i\bm u_i^\top\mathbf e=\mathbf e.
			\end{array}
		\end{equation}		
	\end{thm}
	\begin{proof}
		We first observe that the  centroids in \eqref{eq:kmeans_0} can be expressed as 
		\begin{equation*}
			\bm c_i=\frac{1}{\mathbf e^\top\bm\pi_i}\sum_{n\in[N]}\pi_{in}\bm x_n\qquad\forall i\in[K].
		\end{equation*}
		Substituting these terms into the objective function and expanding the squared norm yield
		\begin{equation*}
			\begin{array}{cl}
				\displaystyle\sum_{i\in[K]}\sum_{n:\pi_{in}=1}\|\bm x_n-\bm c_i\|^2&\displaystyle=\sum_{i\in[K]}\sum_{n\in[N]}\pi_{in}\|\bm x_n-\bm c_i\|^2\\
				&\displaystyle=\left(\sum_{n\in[N]}\|\bm x_{n}\|^2\right)-\left(\sum_{i\in[K]}\frac{1}{\mathbf e^\top\bm \pi_i}\sum_{p,q\in[N]}\pi_{ip}\pi_{iq}\bm x_{p}^\top \bm x_{q}\right)\\[5mm]
				&\displaystyle=\tr(\bm X^\top\bm X)-\sum_{i\in[K]}\frac{1}{\mathbf e^\top\bm \pi_i}\tr(\bm X^\top\bm X\bm \pi_i\bm \pi_i^\top). 
			\end{array}
		\end{equation*}
		Thus, \eqref{eq:kmeans_0} can be {rewritten as}  
		\begin{equation}
			\label{eq:kmeans_2}
			\begin{array}{cll}
				\displaystyle\min&\displaystyle\tr(\bm X^\top\bm X)-\sum_{i\in[K]}\frac{1}{\mathbf e^\top\bm \pi_i}\tr(\bm X^\top\bm X\bm \pi_i\bm \pi_i^\top)\\
				\displaystyle\st&\displaystyle\bm \pi_i\in\{0,1\}^N&\forall i\in[K]\\
				&\displaystyle\mathbf e^\top\bm\pi_i\geq 1&\forall i\in[K]\\
				&\displaystyle\sum_{i\in[K]} \bm{\pi}_i= \mathbf e. 
			\end{array}
		\end{equation}		
		For any feasible solution $(\bm \pi_i)_{i\in[K]}$ to \eqref{eq:kmeans_2} we define the vectors $(\bm u_i)_{i\in[K]}$ that satisfy 
		\begin{equation*}
			\bm u_i=\frac{\bm \pi_i}{\sqrt{\mathbf e^\top\bm\pi_i}}\quad\forall i\in[K]. 
		\end{equation*}
		We argue that the solution  $(\bm u_i)_{i\in[K]}$ is feasible to \ref{eq:kmeans_1} and yields the same objective value. Indeed, we  have
		\begin{equation*}
			\bm u_i^\top\bm u_i=\frac{\bm \pi_i^\top\bm{\pi}_i}{{\mathbf e^\top\bm\pi_i}}=1\quad\forall i\in[K]
		\end{equation*}
		because  $\bm \pi_i\in\{0,1\}^N$ and $\mathbf e^\top\bm\pi_i\geq 1$ for all $i\in[K]$. We also have 
		\begin{equation*}
			\sum_{i\in[K]}\bm u_i\bm u_i^\top\mathbf e=\sum_{i\in[K]}\frac{\bm \pi_i}{\sqrt{\mathbf e^\top\bm\pi_i}}\frac{\mathbf e^\top\bm \pi_i}{\sqrt{\mathbf e^\top\bm\pi_i}}=\mathbf e,
		\end{equation*}
		and
		\begin{equation*}
			\bm u_i^\top\bm u_j=0\quad\forall i,j\in[K]:i\neq j
		\end{equation*}
		since the constraint $\sum_{i\in[K]} \bm{\pi}_i= \mathbf e$ in \eqref{eq:kmeans_2} ensures that each data point is assigned to at most $1$ cluster.  Verifying the objective value of this solution, we obtain 
		\begin{equation*}
			\begin{array}{ccll}
				\displaystyle\tr(\bm X^\top\bm X)-\sum_{i\in[K]}\tr(\bm X^\top\bm X\bm u_i\bm u_i^\top)=\displaystyle \tr(\bm X^\top\bm X)-\sum_{i\in[K]}\frac{1}{\mathbf e^\top\bm \pi_i}\tr(\bm X^\top\bm X\bm \pi_i\bm \pi_i^\top).
			\end{array}
		\end{equation*}		
		Thus, we conclude that problem \ref{eq:kmeans_1} constitutes a relaxation of \eqref{eq:kmeans_2}. 
		
		To show that \ref{eq:kmeans_1} is indeed an exact reformulation, consider  any feasible solution $(\bm u_i)_{i\in[K]}$ to this problem. For any fixed $i,j\in[K]$, the complementary constraint $\bm u_i^\top\bm u_j=0$ in \ref{eq:kmeans_1} means that
		\begin{equation*}
			u_{in}>0\Longrightarrow u_{jn}=0 \quad \text{and} \quad u_{jn}>0\Longrightarrow u_{in}=0 \quad\text{for all}\quad n\in[N].
		\end{equation*}
		Thus, in view of the last constraint in \ref{eq:kmeans_1}, we must have $\bm u_i\in\{0,{1}/{\bm u_i^\top\mathbf e}\}^N$ for every~$i\in[K]$. %The constraint $\bm u_i^\top\bm u_i=1$ in \eqref{eq:kmeans_1} further implies that  $\# \bm u_i>0$ and  $\bm u_i\in\{0,{1}/\sqrt{\#\bm u_i}\}^L$. 
		Using this observation, we define the binary vectors $(\bm \pi_i)_{i\in[K]}$ that satisfy
		\begin{equation*}
			\label{eq:solution_pi}
			\bm \pi_i=\bm u_i\bm u_i^\top\mathbf e\in\{0,1\}^N\qquad\forall i\in[K]. 
		\end{equation*}
		For every $i\in[K]$, we find that $\mathbf e^\top\bm{\pi}_i\geq 1$ since $\bm u_i^\top\bm u_i=1$. Furthermore, we have
		\begin{equation*}
			\sum_{i\in[K]} \bm{\pi}_i= \sum_{i\in[K]} \bm{u}_i\bm u_i^\top\mathbf e=\mathbf e. 
		\end{equation*}
		Substituting the constructed solution $(\bm \pi_i)_{i\in[K]}$  into the objective function of \eqref{eq:kmeans_2}, we obtain
		\begin{equation*}
			\begin{array}{rl}
				\displaystyle\tr(\bm X^\top\bm X)-\sum_{i\in[K]}\frac{1}{\mathbf e^\top\bm \pi_i}\tr(\bm X^\top\bm X\bm \pi_i\bm \pi_i^\top)&=\displaystyle\tr(\bm X^\top\bm X)-\sum_{i\in[K]}\frac{
					(\bm u_i^\top\mathbf e)^2}{\mathbf e^\top\bm u_i\bm u_i^\top\mathbf e}\tr(\bm X^\top\bm X\bm u_i\bm u_i^\top)\\
				&=\displaystyle\tr(\bm X^\top\bm X)-\sum_{i\in[K]}\tr(\bm X^\top\bm X\bm u_i\bm u_i^\top).
			\end{array}
		\end{equation*}
		Thus, any feasible solution to \ref{eq:kmeans_1} can be used to construct a feasible solution to \eqref{eq:kmeans_2} that yields the same objective value. Our previous argument that \eqref{eq:kmeans_2} is a relaxation of~\ref{eq:kmeans_1} then implies that both problems are indeed equivalent. This completes the proof. 
	\end{proof}
	\begin{rem}
		%The equivalent reformulation \eqref{eq:kmeans_1} for $K$-means is reminiscent of the formulation \eqref{eq:NMF1} for ONMF. Here,  we have 
		The  constraint $\sum_{i\in[K]}\bm u_i\bm u_i^\top\mathbf e=\mathbf e$ in \ref{eq:kmeans_1}   ensures that there are no fractional values in the resulting cluster assignment vectors $(\bm{\pi}_i)_{i\in[K]}$. While the formulation \eqref{eq:NMF1} is only applicable for instances of ONMF problem with non-negative input data $\bm X$, the reformulation \ref{eq:kmeans_1} remains valid for any instances of $K$-means clustering problem, even if the input data matrix $\bm X$ contains negative components. 
	\end{rem}
	
	\begin{rem}
		%The equivalent reformulation \eqref{eq:kmeans_1} for $K$-means is reminiscent of the formulation \eqref{eq:NMF1} for ONMF. Here,  we have 
		In \cite[Section 2]{ding2005equivalence} and  \cite[Theorem 1]{li2006relationships}, it was claimed that the ONMF problem \eqref{eq:NMF1} is equivalent to {the} $K$-means clustering problem \eqref{eq:kmeans}.  Theorem \ref{thm:kmeans_1} above amends this result by showing that both problems become equivalent if and only if the constraint $\sum_{i\in[K]}\bm u_i\bm u_i^\top\mathbf e=\mathbf e$  is added to \eqref{eq:NMF1}. 
		%			The ONMF problem in \eqref{eq:NMF1} with the additional constraint $\sum_{i\in[K]}\bm u_i\bm u_i^\top\mathbf e=\mathbf e$   is equivalent to K-means clustering problem given by \eqref{eq:kmeans_1}. This  amends the previous incorrect result by Ding et al.~\cite[Section 2]{ding2005equivalence} and Li and Ding \cite[Theorem 1]{li2006relationships} who claimed that both problems are equivalent.
	\end{rem}
	
	\begin{rem}
		We can reformulate the objective function of \ref{eq:kmeans_1} as 
		$\frac{1}{2}\tr\left(\bm D\sum_{i\in[K]}\bm u_i\bm u_i^\top\right)$,
		where $\bm D$ is the matrix with components $D_{pq}=\|\bm x_p-\bm x_q\|^2$, $p,q\in[N]$. To obtain this reformulation, define $\bm Y=\sum_{i\in[K]}\bm u_i\bm u_i^\top$. Then we have
		\begin{equation*}
			\begin{array}{cll}
				\displaystyle\frac{1}{2}\tr(\bm D\bm Y)&\displaystyle=\frac{1}{2}\sum_{p,q\in[N]}\|\bm x_p-\bm x_q\|^2Y_{pq}\\
				&\displaystyle=\frac{1}{2}\sum_{p,q\in[N]}\left(\bm x_p^\top\bm x_p+\bm x_q^\top\bm x_q-2\bm x_p^\top\bm x_q\right)Y_{pq}\\
				&\displaystyle=\frac{1}{2}\left(2\sum_{p\in[N]}\sum_{q\in[N]}\bm x_p^\top\bm x_pY_{pq}\right)-\sum_{p,q\in[N]}\bm x_p^\top\bm x_qY_{pq}\\
				&\displaystyle =\left(\sum_{p\in[N]}\bm x_p^\top\bm x_p\right)-\left(\sum_{p,q\in[N]}\bm x_p^\top\bm x_qY_{pq}\right)=\displaystyle\tr(\bm X^\top\bm X)-\tr(\bm X^\top\bm X\bm Y). 
			\end{array}
		\end{equation*}
		Here, the fourth equality holds because of the last constraint in \ref{eq:kmeans_1} which ensures that $\sum_{q\in[N]}Y_{pq}=1$ for all~$p\in[N]$. 
	\end{rem}
	We are now well-positioned to derive an equivalent generalized completely positive program for the $K$-means clustering problem. % \blue{To further simplify our notation we will employ the additional set 
	%	\begin{equation*}
	%	\mathcal V^\prime(N,K)=\left\{(\bm V_{ij})_{i,j\in[K]} \in \mathcal V(N,K) :
	%	\sum_{i\in[K]}\bm V_{ii}\mathbf e=\mathbf e\right\}
	%	\end{equation*} in the remainder of this section.} 
	%	
	\begin{thm}
		The following {generalized} completely positive program solves the $K$-means clustering problem:
		\begin{equation}
			\tag{$\overline{\mathcal Z}$}
			\label{eq:kmeans_CPP}
			\begin{array}{ccll}
				Z^\star=&\min&\displaystyle\tr(\bm X^\top\bm X)-\sum_{i\in[K]}\tr(\bm X^\top\bm X\bm V_{ii})\\
				&\st& \left((\bm p_i)_{i\in[K]},(\bm Q_{ij})_{i,j\in[K]}\right)\in	\mathcal W\left(\soc_+^{N+1}\times \RR_+^{N+1},K\right), \; %		&&\bm p_i\in\soc_+^{N+1}\times \RR_+^{N+1},\;\bm Q_{ij}\in \RR_+^{2(N+1)\times 2(N+1)},\;
				(\bm V_{ij})_{i,j\in[K]}\in\mathcal V(N,K)\\
				&& \bm w\in\RR_+^K,\; z_{ij} \in \RR_+\;,\bm u_i,\bm s_i,\bm h_{ij},\bm r_{ij}\in\RR_+^N,\;\bm Y_{ij},\bm G_{ij}\in \RR_+^{N\times N}&\forall i,j\in[K]\\
				&	&\bm p_i=\begin{bmatrix}
					\bm u_i\\ 1 \\ \bm s_i \\ w_i
				\end{bmatrix},\;\bm Q_{ij}=
				\begin{bmatrix}\bm V_{ij}&\bm u_i & \bm G_{ij} & \bm h_{ij}\\ 
					\bm u_j^\top & 1 & \bm s_j^\top & w_j \\
					\bm G_{ji}^\top &\bm s_i & \bm Y_{ij} & \bm r_{ij}\\
					\bm h_{ji}^\top & w_i & \bm r_{ji}^\top & z_{ij}
				\end{bmatrix}
				&\forall i,j\in[K]\\ 
				%	&\tr(\bm V_{ii})=1 &\forall i\in[K]\\ 
				%	&\tr(\bm V_{ij})=0 &\forall i,j\in[K]:i\neq j\\
				&	&\displaystyle\sum_{i\in[K]}\bm V_{ii}\mathbf e=\mathbf e\\
				&	&\diag(\bm V_{ii})=\bm h_{ii},\;\;
				\bm u_{i}+\bm s_i=w_i\mathbf e,\;\;
				\diag(\bm V_{ii}+ \bm Y_{ii}+2\bm G_{ii})+z_{ii}\mathbf e-2\bm h_{ii}-2\bm r_{ii}=\bm 0&\forall i\in[K].
				%&	&\begin{bmatrix}
				%\bm Q_{11}&\cdots & \bm Q_{1K}&\bm p_1\\
				%\vdots &\ddots&\vdots&\vdots\\
				%\bm Q_{K1}&\cdots&\bm Q_{KK}&\bm p_K\\
				%\bm p_{1}^\top&\cdots&\bm p_{K}^\top&1\\
				%\end{bmatrix}\in\mathcal 
				%	&	&\blue{\mathcal{W}(K) \in C\left((\soc_+^{N+1}\times \RR_+^{N+1})^K\times \RR_+\right)}
			\end{array}
		\end{equation}
	\end{thm}
	\begin{proof}
		We consider the following equivalent reformulation of \ref{eq:kmeans_1} with two additional strengthening constraint systems. 
		\begin{equation}
			\label{eq:kmeans_3}
			\begin{array}{cll}
				\min&\displaystyle\tr(\bm X^\top\bm X)-\sum_{i\in[K]}\tr(\bm X^\top\bm X\bm u_i\bm u_i^\top)\\
				\st& \bm U\in\mathcal U(N,K),\;\bm S\in\RR_+^{N\times K},\;\bm w\in\RR_+^K\\
				%	&\bm u_i^\top\bm u_i=1&\forall i\in[K]\\
				%	&\bm u_i^\top\bm u_j=0&\forall i,j\in[K]:i\neq j\\
				&\displaystyle\sum_{i\in[K]}\bm u_i\bm u_i^\top\mathbf e  = \mathbf e \\
				&\bm u_{i}\circ\bm u_i=w_i\bm u_{i}&\forall i\in[K]\\
				&\bm u_{i}+\bm s_i=w_i\mathbf e&\forall i\in[K]
			\end{array}
		\end{equation}
		Since $\bm s_i\geq \bm 0$, the last constraint system in \eqref{eq:kmeans_3} implies that $\bm u_i\leq w_i\mathbf e$, while the penultimate constraint system  ensures that $\bm u_i$ is a binary vector, \ie, $\bm u_i\in\{0,w_i\}^N$ for some $w_i\in\RR_+$. Since any feasible solution to~\ref{eq:kmeans_1} satisfies these conditions, we may thus conclude that the problems  \ref{eq:kmeans_1} and  \eqref{eq:kmeans_3} are indeed equivalent. As we will see below, the exactness of the generalized completely positive programming reformulation is reliant on these two redundant constraint systems. 
		
		{
			We now repeat the same derivation steps as in the proof of Theorem \ref{thm:CPP}. First, we  introduce an auxiliary decision variable $\bm p=(\bm p_i)_{i\in[K]}$, that satisfies
			\begin{align*}
				\bm p_i=\begin{bmatrix}
					\bm u_i\\ t_i \\ \bm s_i \\ w_i
				\end{bmatrix}\in \soc_+^{N+1}\times \RR_+^{N+1}\qquad\forall i\in[K].
			\end{align*}
			We then set $\mathcal K=(\soc_+^{N+1}\times \RR_+^{N+1})^K$, and define the structured feasible sets
			\begin{equation}
				\label{eq:F_0}
				\mathcal F_0=\left\{\bm p\in\mathcal K:\begin{array}{ll}t_i=1&\forall i\in[K]\\
					\bm u_{i}+\bm s_i=w_i\mathbf e&\forall i\in[K]
				\end{array}\right\}, 
			\end{equation}	
			\begin{equation}
				\label{eq:F_1}
				\mathcal F_1=\left\{\bm p\in\mathcal F_0:\begin{array}{ll}
					\bm u_i^\top\bm u_i=1&\forall i\in[K]\\
					\bm u_i^\top\bm u_j=0&\forall i,j\in[K]:i\neq j\\
					\bm u_{i}\circ\bm u_i=w_i\bm u_{i}&\forall i\in[K]
				\end{array}\right\}, 
			\end{equation}		
			and	$\mathcal F_2=\mathcal F=\left\{\bm p\in\mathcal F_1:\sum_{i\in[K]}\bm u_i\bm u_i^\top\mathbf e  = \mathbf e\right\}$.
			Here, we find that the recession cone of $\mathcal F_0$ is given by
			\begin{equation*}
				\mathcal F^\infty=\left\{\bm p\in\mathcal K:\begin{array}{ll}\bm u_i=\bm 0,\;t_i=0&\forall i\in[K]\\
					\bm u_{i}+\bm s_i=w_i\mathbf e&\forall i\in[K]
				\end{array}
				\right\}.
			\end{equation*}
			Next, we set the vector $\overline{\bm p}=(\overline{\bm p}_1,\ldots,\overline{\bm p}_K)\in\mathcal F$ in Theorem \ref{thm:main} to satisfy
			\begin{equation*}
				\overline{\bm p}_i=\begin{bmatrix}
					\overline{\bm u}_i\\ 1\\ \overline{\bm s}_i\\ {\overline{w}_i}
				\end{bmatrix}\in \soc_+^{N+1}\times\RR^{N+1}_+\qquad\forall i\in[K],
			\end{equation*}
			where the subvectors $\{\overline{\bm u}_i\}_{i\in[K]}$,  $\{\overline{\bm s}_i\}_{i\in[K]}$  and {$\{\overline{w}_i\}_{i\in[K]}$}  are chosen so that they are feasible in \eqref{eq:kmeans_3}. In view of the description of the recession cone $\mathcal F^\infty$ and the structure of the quadratic constraints in $\mathcal F$, one can verify that such a vector $\overline{\bm p}$  satisfies the condition {\eqref{eq:p_bar}} in Theorem \ref{thm:main}. 
			%, for all \blue{$d = [\bm 0 \;\ 0 \;\ \bm s_i \;\ w_i]^\top \in \mathcal{F}^\infty$ and by appropriately defining $\bm C_j$ and $\bm c_j$ for all $j \in [J]$.}.  
			
			It remains to show that condition~\eqref{eq:phi_j} is also satisfied. To this end, it is already verified in the proof of Theorem \ref{thm:CPP} that 
			\begin{equation*}
				\max_{\bm p\in\mathcal F_0}\left\{\bm u_i^\top\bm u_i\right\}=1\quad\forall i\in[K]\quad\textup{ and }\quad\min_{\bm p\in\mathcal F_0}\left\{ \bm u_i^\top\bm u_j\right\}=0\quad\forall i,j\in[K]:i\neq j.
			\end{equation*}
			We now show that 
			\begin{equation}
				\label{eq:u_square}
				\min_{\bm p\in\mathcal F_0}\left\{ w_iu_{in}-u_{in}^2\right\}=0\qquad\forall i\in[K]\;\forall n\in[N].
			\end{equation}
			We first demonstrate that the constraint $\bm u_i+\bm s_i=w_i\mathbf e$ in \eqref{eq:F_0} implies  $\bm u_i\circ\bm u_i\leq w_i\bm u_i$. Indeed, since $\bm s_i\geq \bm 0$, we have $w_i\mathbf e-\bm u_i\geq\bm 0$. Applying a componentwise multiplication with the components of  $\bm u_i\geq\bm 0$ on the left-hand side, we arrive at the desired inequality. Thus, we  find that each equation in~\eqref{eq:u_square} indeed holds, where  equality is attained whenever $u_{in}=0$. Finally, we verify that 
			\begin{equation}
				\label{eq:sum_V}
				\min_{\bm p\in\mathcal F_1} \left\{\sum_{i\in[K]} u_{in} \bm u_i^\top\mathbf e\right\}=1\qquad\forall n\in[N].
			\end{equation}
			Note that the constraint $\bm u_i\circ\bm u_i=w_i\bm u_i$ in \eqref{eq:F_1} implies that $\bm u_i\in\{0,w_i\}^N$, while
			the constraint $\bm u_i^\top\bm u_i=1$ further implies that {$\#\bm u_i w_i^2 = 1$}. Moreover, the
			complementary constraint $\bm u_i^\top\bm u_j=0$ ensures that 
			\begin{equation*}
				u_{in} > 0 \Rightarrow u_{jn} = 0 \quad\text{and}\quad u_{jn} > 0 \Rightarrow u_{in} = 0\qquad\forall n\in[N] \;\forall i,j\in[K]:i\neq j.
			\end{equation*}
			%Combining this with the constraint $\bm u_i^\top\bm u_i=1$, we obtain
			Thus, for any feasible vector $\bm p\in\mathcal F_1$, we have
			\begin{equation*}
				\sum_{i\in[K]} u_{in} \bm u_i^\top\mathbf e=\sum_{i\in[K]} u_{in} w_i\#\bm u_i=\sum_{i\in[K]}{\frac{u_{in}}{w_i} = \frac{w_k}{w_k}} = 1,
			\end{equation*} 
			for some $k\in[K]$ such that $u_{kn}=w_k$. Thus, the equalities \eqref{eq:sum_V} indeed hold. In summary, we have shown that all conditions in Theorem \ref{thm:main} are satisfied.

			We now introduce new variables, in addition to the ones described in \eqref{matvar}, that {linearize} the quadratic terms, as follows:
			\begin{align}
				\label{matvar1}
				z_{ij}=w_iw_j,\;\bm h_{ij}=\bm u_iw_j,\;\bm r_{ij}=\bm s_iw_j,\;\bm Y_{ij}=\bm s_i\bm s_j^\top,\;\bm G_{ij}=\bm u_i\bm s_j^\top\;\qquad\forall i,j\in[K].	
			\end{align}
			We further define an auxiliary decision variable $\bm Q_{ij}$, $i, j \in[K]$, that satisfy
			\begin{align*}
				\bm Q_{ij} = \bm p_i \bm p_j^\top = 
				\begin{bmatrix}\bm V_{ij}&\bm u_i & \bm G_{ij} & \bm h_{ij}\\ 
					\bm u_j^\top & 1 & \bm s_j^\top & w_j \\
					\bm G_{ji}^\top &\bm s_i & \bm Y_{ij} & \bm r_{ij}\\
					\bm h_{ji}^\top & w_i & \bm r_{ji}^\top & z_{ij}
				\end{bmatrix}.
			\end{align*}
			Using these new terms, we construct the set $\mathcal R$ in Theorem \ref{thm:main} as follows: 
			\begin{equation*}
				\mathcal R=\left\{\begin{bmatrix}
					\bm Q_{11}&\cdots & \bm Q_{1K}&\bm p_1\\
					\vdots &\ddots&\vdots&\vdots\\
					\bm Q_{K1}&\cdots&\bm Q_{KK}&\bm p_K\\
					\bm p_{1}^\top&\cdots&\bm p_{K}^\top&1\\
				\end{bmatrix} 
				%\begin{bmatrix}\bm Q&\bm p\\\bm p^\top & 1\end{bmatrix}
				\in\mathcal{C}(\mathcal K \times \RR_+):
				\begin{array}{ll}
					\bm p_i=\begin{bmatrix}
						\bm u_i\\ 1 \\ \bm s_i \\ w_i
					\end{bmatrix},\;
					\bm Q_{ij}=
					\begin{bmatrix}\bm V_{ij}&\bm u_i & \bm G_{ij} & \bm h_{ij}\\ 
						\bm u_j^\top & 1 & \bm s_j^\top & w_j \\
						\bm G_{ji}^\top &\bm s_i & \bm Y_{ij} & \bm r_{ij}\\
						\bm h_{ji}^\top & w_i & \bm r_{ji}^\top & z_{ij}
					\end{bmatrix}&\forall i,j\in[K]\\
					\tr(\bm V_{ii}) = 1& \forall i \in [K]\\ 
					\tr(\bm V_{ij}) = 0 & \forall i, j \in [K]: i \neq j	\\
					\sum_{i\in[K]}\bm V_{ii}\mathbf e=\mathbf e \\
					\diag(\bm V_{ii})=\bm h_{ii}, \ \ \bm u_{i}+\bm s_i=w_i\mathbf e & \forall i \in [K]\\
					\diag(\bm V_{ii}+ \bm Y_{ii}+2\bm G_{ii})\\
					\qquad+z_{ii}\mathbf e-2\bm h_{ii}-2\bm r_{ii}=\bm 0 &\forall i\in[K]
				\end{array}\right\}. 
			\end{equation*}
			%			
			%			\begin{equation}
			%			\label{conv2}
			%			\begin{array}{cll}
			%			\mathcal{C^\prime}(\mathcal{L} \cap \mathcal{Q}) = \bigg \{(\bm p, \bm Q) \ | & \begin{bmatrix}\bm Q &\bm p\\ \bm p^\top & 1 \end{bmatrix}   \in\mathcal{C}((\soc^{N+1}_+ \times \RR^{N+1}_+)^K \times \RR_+), \\ 
			%			& (V_{ij})_{i,j \in [K]} \in \mathcal{V}(N, K) \\
			%			& p_{(2i-1)(N+1)} = 1, \ \ p^2_{(2i-1)(N+1)} = 1 \ \ \ \forall i\in[K],  \\
			%			&\displaystyle\sum_{i\in[K]}\bm V_{ii}\mathbf e=\mathbf e\\
			%			&\diag(\bm V_{ii})=\bm h_{ii}, \ \ \bm u_{i}+\bm s_i=w_i\mathbf e \ \ \ \forall i \in [K],\\
			%			&\diag(\bm V_{ii}+ \bm Y_{ii}+2\bm G_{ii})+z_{ii}\mathbf e-2\bm h_{ii}-2\bm r_{ii}=0, \ \ \forall i\in[K] \bigg\}.
			%			\end{array}
			%			\end{equation}	
			Here, the last constraint system arises from squaring the left-hand sides of the equalities
			\begin{equation*}
				u_{in}+s_{in}-w_i=0\qquad\forall i\in[K]\;\forall n\in[N],
			\end{equation*}
			which correspond to the  last constraint system in \eqref{eq:kmeans_3}. 
			Finally by linearizing the objective function using variables in \eqref{matvar} and \eqref{matvar1}, we  arrive at the generalized completely positive program \ref{eq:kmeans_CPP}. This completes the proof. 	}	\end{proof}	
	
	\section{Approximation Algorithm for $K$-means Clustering}
	\label{sec:kmeans_algo}
	In this section, we develop a new approximation algorithm for  $K$-means clustering. To this end, we observe that in the reformulation \ref{eq:kmeans_CPP} the difficulty of {the} original problem is now entirely absorbed in the completely positive cone $\mathcal C(\cdot)$ which has been well studied in the literature~\cite{bomze2002solving,burer2012copositive,DKP02:copositive}. Any such completely positive program admits the hierarchy of increasingly accurate  SDP relaxations that are obtained by replacing  the  cone $\mathcal C(\cdot)$ with  progressively tighter semidefinite-representable outer approximations~\cite{DKP02:copositive,lasserre2009convexity,parrilo2000structured}. 
	%the cone 
	%\begin{equation*}
	%\mathcal C^0\left((\soc_+^{N+1}\times \RR_+^{N+1})^K\times \RR\right)=\left\{\bm M\in\mathbb S^{2K(N+1)+1}:\bm M\succeq \bm 0,\;\bm M\geq \bm 0,\;\tr(\mathbb J\bm M)\geq 0\right\},
	%\end{equation*} 
	%where $\mathbb J=\diag(1,\mathbf e,\bm 0,0,\cdots,1,\mathbf e,\bm 0,0,0)\in\RR^{2(N+1)K+1}$. 
	{For the {generalized} completely positive program \ref{eq:kmeans_CPP}, we employ the simplest outer approximation % and, by a the size of the resulting semidefinite program is further reduced to alleviate the intractability. 
		%	The relaxation i
		that is obtained by replacing the completely positive cone $\mathcal C\left((\soc_+^{N+1} \times \RR_+^{N+1})^K\times \RR_+\right)$ in \ref{eq:kmeans_CPP} with 
		its coarsest outer approximation~\cite{sturm2003cones}, given by the cone}
	\begin{equation*}
		\left\{\bm M\in\mathbb S^{2K(N+1)+1}:\bm M\succeq \bm 0,\;\bm M\geq \bm 0,\;\tr(\mathbb J_i\bm M)\geq 0\;\; i\in[K]\right\},
	\end{equation*} 
	where 
	\begin{equation*}\begin{array}{c}
			\mathbb J_1=\diag\left([-\mathbf e^\top,1,\bm 0^\top,0,\cdots,\bm 0^\top,0,0]^\top\right),\\
			\mathbb J_2=\diag\left([\bm 0^\top,0,-\mathbf e^\top,1,\cdots,\bm 0^\top,0,0]^\top\right),\\
			\cdots\\
			\mathbb J_K=\diag\left([\bm 0^\top,0,-\mathbf 0^\top,0,\cdots,\mathbf e^\top,1,0]^\top\right).
		\end{array}
	\end{equation*}
	If $\bm M$ has the structure of the large matrix in \ref{eq:kmeans_CPP} then the constraint $\tr(\mathbb J_i\bm M)\geq 0$  reduces to $\tr(\bm V_{ii})\leq 1$, which is redundant and can safely be omitted in view of the stronger equality constraint $\tr(\bm V_{ii})= 1$ in \ref{eq:kmeans_CPP}. In this case, the outer approximation can be simplified to the cone of doubly non-negative matrices given by 
	\begin{equation*}
		\left\{\bm M\in\mathbb S^{2K(N+1)+1}:\bm M\succeq \bm 0,\;\bm M\geq \bm 0\right\}. 
	\end{equation*}
	%omitted the variables $\bm w\in\RR_+^K,\; z_{ij} \in \RR_+\;,\bm s_i,\bm h_{ij},\bm r_{ij}\in\RR_+^N,\;\bm Y_{ij},\bm G_{ij}\in \RR_+^{N\times N}$, $i,j\in[K]$, from \eqref{eq:kmeans_CPP}. 
	%To this end, we first derive a tractable semidefinite programming (SDP) relaxations to~\eqref{eq:kmeans_1}. This is achieved by replacing the completely positive cone $\mathcal C\left((\soc_+^{N+1})^K\times \RR\right)$ in \eqref{eq:NMF_CPP} with the cone of doubly non-negative matrices given by
	%$\left\{\bm M\in\mathbb S^{K(N+1)+1}_+:\bm M\geq \bm 0\right\}$.
	{To further improve computational tractability, we  relax the large semidefinite constraint into  a simpler system of $K$ semidefinite constraints. We summarize our formulation in the following proposition. } 
	%	\blue{We define a new set to further simply our notations,   
	%		\begin{equation*}
	%		\mathcal V^{\prime\prime}(N,K)=\left\{(\bm V_{i})_{i\in[K]}\in\RR_+^{N^2\times K}:\begin{array}{cllll}
	%		\tr(\bm V_{i})=1 &\forall i\in[K],&
	%		\sum_{i\in[K]}\bm V_{i}\mathbf e=\mathbf e \\
	%		\end{array}\right\}
	%		\end{equation*}
	%		and employ this set in the remainder of this section.}
	
	\begin{prop}
		The optimal value of the following SDP constitutes a lower bound on $Z^\star$.
		\begin{equation}
			\tag{${\mathcal R}_0$}
			\label{eq:kmeans_SDP}
			\begin{array}{ccll}
				R_0^\star=&\min&\displaystyle\tr(\bm X^\top\bm X)-\sum_{i\in[K]}\tr(\bm X^\top\bm X\bm V_{i})\\
				&\st& \bm p_i\in\soc_+^{N+1}\times \RR_+^{N+1},\;\bm Q_{i}\in \RR_+^{2(N+1)\times 2(N+1)},\;\bm V_{i}\in \mathbb R_+^{N\times N}&\forall i\in[K]\\
				&& w_i\in\RR_+,\; z_{i} \in \RR_+\;,\bm u_i,\bm s_i,\bm h_{i},\bm r_{i}\in\RR_+^N,\;\bm Y_{i},\bm G_{i}\in \RR_+^{N\times N}&\forall i\in[K]\\
				&	&\bm p_i=\begin{bmatrix}
					\bm u_i\\ 1 \\ \bm s_i \\ w_i
				\end{bmatrix},\;\bm Q_{i}=
				\begin{bmatrix}\bm V_{i}&\bm u_i & \bm G_{i} & \bm h_{i}\\ 
					\bm u_i^\top & 1 & \bm s_i^\top & w_i \\
					\bm G_{i}^\top &\bm s_i & \bm Y_{i} & \bm r_{i}\\
					\bm h_{i}^\top & w_i & \bm r_{i}^\top & z_{i}
				\end{bmatrix}
				&\forall i\in[K]\\ 
				%	&\tr(\bm V_{ii})=1 &\forall i\in[K]\\ 
				%	&\tr(\bm V_{ij})=0 &\forall i,j\in[K]:i\neq j\\
				&	&\displaystyle\sum_{i\in[K]}\bm V_{i}\mathbf e=\mathbf e\\
				&	&\tr(\bm V_{i})=1,\;\;
				\diag(\bm V_{i})=\bm h_{i},\;\;
				\bm u_{i}+\bm s_i=w_i\mathbf e&\forall i \in [K]\\
				& &\diag(\bm V_{i}+ \bm Y_{i}+2\bm G_{i})+z_{i}\mathbf e-2\bm h_{i}-2\bm r_{i}=\bm 0&\forall i\in[K]\\
				& & \mathbf e_1^\top\bm V_{1}\mathbf e=1 \\
				&	&\begin{bmatrix}
					\bm Q_{i}& \bm p_i\\
					\bm p_{i}^\top &1\\
				\end{bmatrix} \succeq \bm 0 &\forall i \in [K]
			\end{array}
		\end{equation}
	\end{prop}
	
	\begin{proof}
		Without loss of generality,  we can assign the first data point $\bm x_1$ to the first cluster. 
		The argument in {the} proof of Theorem \ref{thm:kmeans_1} indicates that the assignment vector for the first cluster is given by
		\begin{equation*}
			\bm\pi_1=\bm u_1\bm u_1^\top\mathbf e=\bm V_{11}\mathbf e.
		\end{equation*} 
		Thus, the data point $\bm x_1$ is assigned to the first cluster if and only if the first element of $\bm\pi_1$ is equal to~$1$, \ie, $1=\mathbf e_1^\top\bm \pi_1=\mathbf e_1^\top\bm V_{11}\mathbf e$. Henceforth, we shall add this constraint to  \ref{eq:kmeans_CPP}. While the constraint is redundant for the completely positive program  \ref{eq:kmeans_CPP}, it will cut-off any symmetric solution in the resulting SDP relaxation. 
		
		We now replace the {generalized} completely positive cone in \ref{eq:kmeans_CPP} with the corresponding cone of doubly non-negative matrices, which yields the following SDP relaxation:
		\begin{equation}
			\label{eq:kmeans_SDP_big}
			\begin{array}{ccll}
				&\min&\displaystyle\tr(\bm X^\top\bm X)-\sum_{i\in[K]}\tr(\bm X^\top\bm X\bm V_{ii})\\
				&\st& {\bm p_i\in\soc_+^{N+1}\times \RR_+^{N+1},\;\bm Q_{ij}\in \RR_+^{2(N+1)\times 2(N+1)},\;{\bm V_{ij}\in\RR_+^{N\times N}}}&\forall i,j\in[K]\\
				&& {\bm w\in\RR_+^K,\; z_{ij} \in \RR_+\;,\bm u_i,\bm s_i,\bm h_{ij},\bm r_{ij}\in\RR_+^N,\;\bm Y_{ij},\bm G_{ij}\in \RR_+^{N\times N}}&\forall i,j\in[K]\\
				&	&{\bm p_i=\begin{bmatrix}
						\bm u_i\\ 1 \\ \bm s_i \\ w_i
					\end{bmatrix},\;\bm Q_{ij}=
					\begin{bmatrix}\bm V_{ij}&\bm u_i & \bm G_{ij} & \bm h_{ij}\\ 
						\bm u_j^\top & 1 & \bm s_j^\top & w_j \\
						\bm G_{ji}^\top &\bm s_i & \bm Y_{ij} & \bm r_{ij}\\
						\bm h_{ji}^\top & w_i & \bm r_{ji}^\top & z_{ij}
				\end{bmatrix}}
				&\forall i,j\in[K]\\ 
				&&\tr(\bm V_{ii})=1 &\forall i\in[K]\\ 
				&&\tr(\bm V_{ij})=0 &\forall i,j\in[K]:i\neq j\\
				%&	&\displaystyle\sum_{i\in[K]}\bm V_{ii}\mathbf e=\mathbf e\\
				&	&{\diag(\bm V_{ii})=\bm h_{ii},\;\;
					\bm u_{i}+\bm s_i=w_i\mathbf e,\;\;
					\diag(\bm V_{ii}+ \bm Y_{ii}+2\bm G_{ii})+z_{ii}\mathbf e-2\bm h_{ii}-2\bm r_{ii}=\bm 0}&\forall i\in[K]\\
				& &\mathbf e_1^\top\bm V_{11}\mathbf e = 1\\
				&	&\begin{bmatrix}
					\bm Q_{11}&\cdots & \bm Q_{1K}&\bm p_1\\
					\vdots &\ddots&\vdots&\vdots\\
					\bm Q_{K1}&\cdots&\bm Q_{KK}&\bm p_K\\
					\bm p_{1}^\top&\cdots&\bm p_{K}^\top&1\\
				\end{bmatrix}\succeq \bm 0
				%		\end{array}
				%		& &\blue{\mathcal{W}(K) \succeq 0.}
			\end{array}
		\end{equation}
		{Since all principal submatrices of the {large} matrix %\blue{delete--- on the left-hand side of  the final constraint}  
			are {also} positive semidefinite, we can further relax the constraint to a more  tractable system
			\begin{equation*}
				\begin{bmatrix}
					\bm Q_{ii}& \bm p_i\\
					\bm p_{i}^\top &1\\
				\end{bmatrix} \succeq \bm 0 \qquad\forall i \in [K]. 
			\end{equation*}
			Next, we eliminate the constraints $\tr(\bm V_{ij})=0$,  $i,j\in[K]:i\neq j$, from \eqref{eq:kmeans_SDP_big}. As the other constraints and the objective function  in the resulting formulation do not involve the decision variables~$\bm V_{ij}$ and $\bm Q_{ij}$,  for any $i,j\in[K]$ such that $i\neq j$, we can  safely omit these decision variables. 
			%any such solution would remain feasible and yield the same objective value if we set $\bm Q_{ij}=\bm 0$.  
			Finally, by renaming all  double subscript variables, 
			\eg,  $\bm Q_{ii}$ to $\bm Q_{i}$, we arrive at the desired semidefinite program~\ref{eq:kmeans_SDP}. This completes the proof. 
		}
	\end{proof}
	\noindent The symmetry breaking constraint $\mathbf e_1^\top\bm V_{1}\mathbf e=1$ in \ref{eq:kmeans_SDP} ensures that the solution $\bm V_{1}$ will be different from any of the solutions $\bm V_{i}$, $i\geq 2$. Specifically, the constraint $\sum_{i\in[K]}\bm V_{i}\mathbf e=\mathbf e$ in \ref{eq:kmeans_SDP} along with the aforementioned symmetry breaking constraint implies that $\mathbf e_1^\top\bm V_{i}\mathbf e=0$  for all $i\geq 2$. Thus, any rounding scheme that identifies the clusters using the solution $(\bm V_i)_{i\in[K]}$ will always assign the data point $\bm x_1$ to the first cluster. It can be shown, however, that there exists a \emph{partially} symmetric optimal solution to \ref{eq:kmeans_SDP} with $\bm V_2=\dots=\bm V_K$. This enables us to derive a further simplification to \ref{eq:kmeans_SDP}. 
	
	\begin{coro}
		\label{coro:kmeans_SDP_simple}
		Problem \ref{eq:kmeans_SDP} is equivalent to the semidefinite program given by
		\begin{equation}
			\label{eq:kmeans_SDP_simple}
			\tag{$\overline{\mathcal R}_0$}
			\begin{array}{ccll}
				R_0^\star=&\min&\displaystyle\tr(\bm X^\top\bm X)-\tr(\bm X^\top\bm X\bm W_1)-\tr(\bm X^\top\bm X\bm W_2)\\
				&\st& \bm \alpha_i\in\soc_+^{N+1}\times \RR_+^{N+1},\;\bm \Gamma_{i} \in \RR_+^{2(N+1)\times 2(N+1)},\;\bm W_i\in \RR_+^{N\times N}&\forall i=1,2\\
				&& \rho_i\in\RR_+,\; \beta_i\in \RR_+\;,\bm \gamma_i, \bm \eta_i,\bm \psi_i ,\bm \theta_i \in\RR_+^N,\;\bm \Sigma_i,\bm \Theta_i \in \RR_+^{N\times N}&\forall i=1,2\\
				&	&\bm \alpha_i=\begin{bmatrix}
					\bm \gamma_i\\ 1 \\ \bm \eta_i \\ \rho_i
				\end{bmatrix},\;\bm \Gamma_{i}=
				\begin{bmatrix}\bm W_{i}&\bm \gamma_i& \bm \Theta_{i} & \bm \psi_{i}\\ 
					\bm \gamma_i^\top & 1 & \bm \eta_i^\top & \rho_i \\
					\bm \Theta_{i}^\top &\bm \eta_i & \bm \Sigma_{i} & \bm \theta_{i}\\
					\bm \psi_{i}^\top & \rho_i & \bm \theta_{i}^\top & \beta_{i}
				\end{bmatrix}&\forall i=1,2\\ 
				&&\tr(\bm W_1)=1,\;\tr(\bm W_2)=K-1\\ 
				&&\diag(\bm W_{i})=\bm \psi_{i},\;\bm \gamma_i+\bm \eta_i=\rho_i\mathbf e,\;\diag(\bm W_{i}+ \bm \Sigma_{i}+2\bm \Theta_{i})+\beta_{i}\mathbf e-2\bm \psi_{i}-2\bm \theta_{i}=\bm 0&
				\forall i=1,2 \\
				&&\displaystyle\bm W_{1}\mathbf e+\bm W_2\mathbf e=\mathbf e	\\
				&&\mathbf e_1^\top\bm W_1\mathbf e=1\\			
				&&\begin{bmatrix}
					\bm \Gamma_{1}& \bm \alpha_1\\
					\bm \alpha_{1}^\top &1\\
				\end{bmatrix} \succeq \bm 0,\;\begin{bmatrix}
					\bm \Gamma_{2}& \bm \alpha_2\\
					\bm \alpha_{2}^\top &K-1\\
				\end{bmatrix} \succeq \bm 0.
			\end{array}
		\end{equation}
	\end{coro}	
	\begin{proof}	{%Since sum of positive semidefinite matrices is also positive semidefinite, 
			Any feasible solution to \ref{eq:kmeans_SDP} can be used to construct a feasible solution to \ref{eq:kmeans_SDP_simple} with the same objective value, as follows:
			\begin{align*}
				&\bm \alpha_1 = \bm p_1,\quad\bm \alpha_2=\sum_{i=2}^K\bm p_i,\quad\bm \Gamma_1=\bm Q_1,\quad\bm \Gamma_2=\sum_{i=2}^K\bm Q_i.
			\end{align*}
			%			\begin{align*}
			%			&\bm \alpha_1 = \bm p_1,\quad\bm \alpha_2=\sum_{i=2}^K\bm p_i,\;\; \bm \gamma_1 = \bm u_1,\quad\bm \gamma_2=\sum_{i=2}^K\bm u_i,\;\; \bm \eta_1 = \bm s_1,\quad\bm \eta_2=\sum_{i=2}^K\bm s_i,\;\; \bm \rho_1 = \bm w_1,\quad\bm \rho_2=\sum_{i=2}^K\bm w_i,\\
			%			&\bm \psi_1 = \bm h_1,\quad\bm \psi_2=\sum_{i=2}^K\bm h_i,\;\; \bm \beta_1 = \bm z_1,\quad\bm \beta_2=\sum_{i=2}^K\bm z_i,\;\; \bm \theta_1 = \bm r_1,\quad\bm \theta_2=\sum_{i=2}^K\bm r_i,\\
			%			&\bm W_1=\bm V_1,\quad\bm W_2=\sum_{i=2}^K\bm V_i,\;\;\bm \Sigma_1=\bm Y_1,\quad\bm \Sigma_2=\sum_{i=2}^K\bm Y_i,\;\;\bm \Theta_1=\bm G_1,\quad\bm \Theta_2=\sum_{i=2}^K\bm G_i,\;\;\bm \Gamma_1=\bm Q_1,\quad\bm \Gamma_2=\sum_{i=2}^K\bm Q_i.
			%			\end{align*}
			Conversely, any feasible solution to \ref{eq:kmeans_SDP_simple} can also  be used to construct a feasible solution to \ref{eq:kmeans_SDP} with the same objective value:
			\begin{equation*}
				\bm p_1 = \bm \alpha_1,\quad\bm p_i=\frac{1}{K-1}\bm\alpha_2,\quad\bm Q_1=\bm \Gamma_1,\quad\bm Q_i=\frac{1}{K-1}\bm \Gamma_2\qquad\forall i=2,\dots,K. 
			\end{equation*}
			%			This is true because $\begin{bmatrix}\bm \Gamma_{2}& \bm \alpha_2\\
			%			\bm \alpha_{2}^\top &1\\
			%		\end{bmatrix} \succeq 0$ implies $\frac{1}{K-1}\begin{bmatrix}
			%			\bm \Gamma_{2}& \bm \alpha_2\\
			%			\bm \alpha_{2}^\top &1\\
			%			\end{bmatrix} \succeq 0$. 
			Thus, the claim follows.
		}
	\end{proof}
	{
		By eliminating the constraints $\diag(\bm W_{i})=\bm \psi_{i},\;\bm \gamma_i+\bm \eta_i=\rho_i\mathbf e,\;\diag(\bm W_{i}+ \bm \Sigma_{i}+2\bm \Theta_{i})+\beta_{i}\mathbf e-2\bm \psi_{i}-2\bm \theta_{i}=0$, $i=1,2$, from \ref{eq:kmeans_SDP_simple} we obtain an even simpler SDP relaxation. 
		\begin{coro}
			The optimal value of the following SDP constitutes a lower bound on $R_0^\star$:
			\begin{equation}
				\tag{${\mathcal R}_1$}
				\label{eq:kmeans_SDP_simple_1}
				\begin{array}{ccll}
					R_1^\star=&\min&\displaystyle\tr(\bm X^\top\bm X)-\tr(\bm X^\top\bm X\bm W_1)-\tr(\bm X^\top\bm X\bm W_2)\\
					&\st& \bm W_1,\bm W_2\in \RR_+^{N\times N}\\
					&&\tr(\bm W_1)=1,\;\tr(\bm W_2)=K-1\\ 
					&&\displaystyle\bm W_{1}\mathbf e+\bm W_2\mathbf e=\mathbf e	\\
					&&\bm W_1\succeq\bm 0,\;\bm W_2\succeq\bm 0\\
					&&\mathbf e_1^\top\bm W_1\mathbf e=1
				\end{array}
			\end{equation}	
		\end{coro}	
	}
	%As it is a relaxation, the optimal value of the semidefinite program \eqref{eq:kmeans_SDP_simple_1} constitutes a lower bound on the optimal value of the true problem \eqref{eq:kmeans_1}. 
	\noindent We remark that the  formulation \ref{eq:kmeans_SDP_simple_1} is reminiscent of the well-known SDP relaxation for  $K$-means clustering~\cite{peng2007approximating}:
	\begin{equation}
		\tag{$\mathcal R_2$}
		\label{eq:kmeans_SDP_0}
		\begin{array}{ccll}
			R_2^\star=&\min&\displaystyle\tr(\bm X^\top\bm X)-\tr(\bm X^\top\bm X\bm Y)\\
			&\st& \bm Y\in \RR_+^{N\times N}\\
			&&\tr(\bm Y)=K\\
			&&\bm Y\mathbf e=\mathbf e\\
			&&\bm Y\succeq\bm 0.
		\end{array} 
	\end{equation}
	{We now derive an ordering of the optimal values of problems \ref{eq:kmeans_1}, \ref{eq:kmeans_SDP},  \ref{eq:kmeans_SDP_simple_1}, and \ref{eq:kmeans_SDP_0}. 
		\begin{thm}
			\label{thm:SDP_relations}
			We have
			\begin{equation*}
				Z^\star\geq R_0^\star\geq R_1^\star\geq R_2^\star.
			\end{equation*}
		\end{thm}
		\begin{proof}
			The first and the second inequalities hold by construction. {To prove the third inequality, consider any feasible solution $(\bm W_1,\bm W_2)$ to \ref{eq:kmeans_SDP_simple_1}. Then, the solution $\bm Y=\bm W_1+\bm W_2$ is feasible to \ref{eq:kmeans_SDP_0} and yields the same objective value, which completes the proof. }
		\end{proof}
		Obtaining any estimations of the best cluster assignment using optimal solutions of problem \ref{eq:kmeans_SDP_0} is a non-trivial endeavor. If we have \emph{exact recovery}, \ie, $Z^\star=R_2^\star$, then an optimal solution of \ref{eq:kmeans_SDP_0} assumes the form
		\begin{equation}
			\label{eq:exact_recovery_solution}
			\bm Y=\sum_{i\in[K]}\frac{1}{\mathbf e^\top\bm{\pi}_i}\bm{\pi}_i\bm{\pi}_i^\top,
		\end{equation}
		where $\bm \pi_i$ is the assigment vector for the $i$-th cluster. Such a solution $\bm Y$ allows for an easy identification of the clusters. If there is no exact recovery then a few additional steps need to be carried out. In \cite{peng2007approximating}, an approximate cluster assignment is obtained by solving exactly another $K$-means clustering problem on a lower dimensional data set whose computational complexity scales with $\mathcal O(N^{(K-1)^2})$. If the solution of the SDP relaxation \ref{eq:kmeans_SDP_0} is close to the exact recovery solution  \eqref{eq:exact_recovery_solution}, then the columns of the matrix $\bm Y\bm X$ will comprise \emph{denoised} data points that are near to the respective optimal cluster centroids. In \cite{mixon2016clustering}, this strengthened signal is leveraged to identify the clusters of the original data points.
		
		The promising result portrayed in Theorem \ref{thm:SDP_relations} implies that any well-constructed rounding scheme that utilizes the improved formulation \ref{eq:kmeans_SDP} (or \ref{eq:kmeans_SDP_simple_1}) will never generate inferior cluster assignments to the ones from  schemes that employ the formulation \ref{eq:kmeans_SDP_0}. 
		Our new SDP relaxation further inspires us to devise an improved approximation algorithm for the $K$-means clustering problem. The central idea of the algorithm is to construct high quality estimates of the cluster assignment vectors $(\bm{\pi}_i)_{i\in[K]}$ using the solution~$(\bm V_{i})_{i\in[K]}$, as follows:
		\begin{equation*}
			\bm\pi_i=\bm V_i\mathbf e\qquad\forall i\in[K]. 
		\end{equation*}
		To eliminate any symmetric solutions, the algorithm gradually introduces symmetry breaking constraints $\mathbf e_{n_i}^\top\bm V_i\mathbf e=1$, $i\geq 2$, to \ref{eq:kmeans_SDP}, where the indices $n_i$, $i\geq 2$, are chosen judiciously. 
		The main component of the algorithm runs in $K$ iterations and proceeds as follows. It first solves the problem \ref{eq:kmeans_SDP} and records its optimal solution~$(\bm V^\star_i)_{i\in[K]}$. In each of the subsequent iterations $k=2,\dots,K$, {the} algorithm  identifies the best unassigned data point $\bm x_n$ for the $k$-th cluster. Here, the best data point corresponds to the index $n$ that maximizes the quantity $\mathbf e_n^\top\bm V^\star_k\mathbf e$. For this index $n$, {the} algorithm then appends the constraint  $\mathbf e_n^\top\bm V^\star_k\mathbf e=1$ to the problem \ref{eq:kmeans_SDP}, which breaks any symmetry in the solution $(\bm V_i)_{i\geq k}$. The algorithm then solves the augmented problem and proceeds to the next iteration. At the end of {the} iterations, the algorithm assigns each data point $\bm x_n$ to the cluster $k$ that maximizes the quantity $\mathbf e_n^\top\bm V^\star_k\mathbf e$. The algorithm concludes with a single step of  Lloyd's algorithm. A summary of the overall procedure is given in Algorithm \ref{alg:kmeans}. 
		\begin{algorithm}[h!]
			\caption{Approximation Algorithm for $K$-Means Clustering}
			\begin{algorithmic}
				\STATE \textbf{Input:} Data matrix $\bm X \in \RR^{D \times N}$ and number of clusters $K$.
				\STATE \textbf{Initialization:} Let $\bm V^\star_i=\bm 0$ and $\mathcal P_i=\emptyset$ for all $i=1,\dots,K$, and $n_k=0$ for all $k =2,\dots,K$.  %$\mathcal S_k=\emptyset$ for all $=2,\dots,K$. 
				\STATE Solve the semidefinite program \ref{eq:kmeans_SDP} with input $\bm X$ and $K$. Update $(\bm V^\star_i)_{i\in[K]}$ with the current  solution. 
				\FOR {$k=2,\ldots,K$}
				\STATE Update $\displaystyle n_k=\argmax_{n\in[N]}\mathbf e_n^\top\bm V^\star_k\mathbf e$. {Break ties arbitarily}. % and  $\mathcal S_k=\left\{n_k\right\}$.  %Set $\displaystyle n^\star=\argmax_{n\in[N]}\mathbf e_n^\top\bm V^\star_k\mathbf e$ and update $\mathcal S_k=\left\{n^\star\right\}$. 
				\STATE Append the constraints $\mathbf e_{n_i}^\top\bm V_i\mathbf e=1$ $\forall i=2,\dots,k$ to the problem \ref{eq:kmeans_SDP}. % $\forall n\in\mathcal S_i$, to the problem \eqref{eq:kmeans_SDP}. 
				\STATE Solve the resulting SDP with input $\bm X$ and $K$. Update $(\bm V^\star_i)_{i\in[K]}$. % with the current  solution. %Solve \eqref{eq:kmeans_SDP} with input $\bm X$. 
				\ENDFOR
				%\STATE \hrulefill
				\FOR {$n=1,\ldots,N$}
				\STATE Set $k^\star=\displaystyle\argmax_{k\in[K]}\mathbf e_n^\top\bm V^\star_k\mathbf e$ and update $\mathcal P_{k^\star}=\mathcal P_{k^\star}\cup\{n\}$. {Break ties arbitarily.}
				\ENDFOR
				%			\\\hrulefill
				\STATE Compute the centroids $\bm c_k=\frac{1}{|\mathcal P_k|}\sum_{n\in\mathcal P_k}\bm x_n$ for all $k=1,\dots,K$. 
				\STATE Reset $\mathcal P_k=\emptyset$  for all $k=1,\dots,K$.
				\FOR {$n=1,\ldots,N$}
				\STATE Set $k^\star=\displaystyle\argmin_{k\in[K]}\|\bm x_n-\bm c_k\|$ and update $\mathcal P_{k^\star}=\mathcal P_{k^\star}\cup\{n\}$. {Break ties arbitarily.}  
				\ENDFOR
				\STATE \textbf{Output:} Clusters $\mathcal P_1,\dots,\mathcal P_K$. 
			\end{algorithmic}
			\label{alg:kmeans}
		\end{algorithm}

		\section{Numerical Results} 
		\label{sec:experiments}
		In this section, we assess the performance of {the}   algorithm described in Section~\ref{sec:kmeans_algo}. All optimization problems are solved with MOSEK v8 using the YALMIP interface \cite{yalmip} on a 16-core 3.4 GHz computer with 32 GB RAM. 
		
		We compare the performance of  Algorithm \ref{alg:kmeans} with {the} classical Lloyd's algorithm and the approximation algorithm\footnote{MATLAB implementation of the algorithm is available at \url{https://github.com/solevillar/kmeans_sdp}.} proposed in \cite{mixon2016clustering} on $50$ randomly generated instances of the $K$-means clustering problem. While our proposed algorithm employs the improved formulation \ref{eq:kmeans_SDP} to identify the clusters, the algorithm in  \cite{mixon2016clustering} utilizes the existing SDP relaxation \ref{eq:kmeans_SDP_0}. 
		
		{
			We adopt the setting of \cite{awasthi2015relax} and consider $N$  data points in $\RR^D$  supported on $K$ balls of the same radius~$r$.  We set $K=3$, $N=75$, and   $r=2$, and run the experiment for $D=2,\dots,6$. All  results are averaged over $50$ trials generated as follows. In each trial, we set the centers of the balls to $\bm 0$, ${\mathbf e}/{\sqrt D}$, and ${c\mathbf e}/{\sqrt D}$, where  the scalar $c$ is drawn uniformly at random from interval $[10, 20]$. This setting ensures that the first two balls are always separated by unit distance irrespective of $D$, while the third ball is placed further with a  distance $c$ from the origin. Next, we sample  $N/ K$ points uniformly at random from each ball. The resulting $N$ data points are then input to the three  algorithms. 
		}
		%and sample $N/K$ points on each ball uniformly at random. In the experiment we set $K=3$ and $N=75$, and further fix the radius of each ball to $2$. 

		%	 We fix the number of clusters to $K = 3$. In each trial, we uniformly and independently draw $N/ K$ random points from $K$ balls in $\RR^D$. Here, the integer $N$ is set to 75. We fix the ball radius to 2 and centers to $[0]^D, [1/\sqrt D]^D \text{ and } [c/\sqrt D]^D$, where the scalar $c$ is drawn uniformly at random from $[10, 20]$. We choose such a setting to ensure that the first two balls overlap largely and are always separated by unit distance irrespective of $D$. We also place the third ball significantly far away from first two balls but separated by same distance irrspective of $D$. 
		
		Table \ref{tab:gaps} reports the quality of  cluster assignments generated from  Algorithm~\ref{alg:kmeans} relative to the ones generated from the algorithm in \cite{mixon2016clustering} and the Lloyd's algorithm. {The mean in the table represents average percentage improvement of the true objective value from Algorithm~\ref{alg:kmeans} relative to other algorithms. The $p$th percentile is the value below which $p\%$ of these improvements  may be found. %\footnote{\url{https://en.wikipedia.org/wiki/Percentile}.}} 
			We find that our proposed algorithm significantly outperforms both the other algorithms in view of the mean and the $95$th percentile statistics.   We further observe that the improvements deteriorate as problem dimension $D$ increases. This should be expected as the clusters become more apparent in a higher dimension, which makes them easier to be identified by all the algorithms. %which makes easier to identify by all the algorithms.} 
		The percentile statistics further indicate that while the other algorithms can  generate extremely poor cluster assignments, our algorithm consistently produces high quality cluster assignments and rarely loses by more than $5\%$. 
		\begin{table}[h!]
			\color{black}
			\centering
			\begin{tabular}{c|c|c|c|c|c|c|c|}
				\multicolumn{2}{c}{}&\multicolumn{6}{c}{Statistic} \\\cline{3-8}
				\multicolumn{1}{c}{}	&&
				\multicolumn{2}{c|}{Mean} &
				\multicolumn{2}{c|}{5th Percentile} &
				\multicolumn{2}{c|}{95th Percentile} \\[0.0mm] \cline{2-8}
				\multirow{ 5}{*}{$D$} 
				&2 &  47.4\% & 26.6\% &  -4.4\% & 17.6\% &  186.7\% & 36.5\%  \\
				&3 & 21.3\% & 18.3\%  & -2.3\% & 10.9\% & 168.9\%  & 25.5\%  \\
				&4 & 5.7\% & 14.5\%  & -1.5\% & 9.5\% & 10.8\%  & 20.8\% \\
				&5 & 9.5\% & 11.1\%  & -2.1\% & 7.3\% & 125.8\%  & 14.5\%  \\
				&6 & 4.8\% & 10.9\%  & -0.7\% & 7.5\% & 8.4\%  & 13.8\% \\
				\cline{2-8}
				\cline{2-8}
			\end{tabular}
			\caption{\color{black} Improvement of the true $K$-means objective value of the cluster assignment generated from the Algorithm~\ref{alg:kmeans} relative to the ones generated from the algorithm in \cite{mixon2016clustering} (left) and the Lloyd's Algorithm (right). 
				\label{tab:gaps}}
		\end{table}	%
		
		{
			\paragraph{Acknowledgements.}The authors are thankful to the Associate Editor and two anonymous referees for their constructive comments and suggestions that led to substantial improvements of the paper. This research was supported by the National Science Foundation grant no.~1752125.
		}
		
		%% Here starts the e-companion (EC)
		%%%%%%%%%%%%%%%%%%%%%%%%%%%%%%%%%%%%%%%%%%%%%%%%%%%%%%%%%%
		%\ECSwitch
		
		%\ECDisclaimer
		%%%%%%%%%%%%%%%%%%%%%%%%%%%%%%%%%%%%%%%%%%%%%%%%%%%%%%%%%%
		
		%%% Main head for the e-companion
		%\ECHead{E-Companion}
		
		% References here (outcomment the appropriate case)
		
		% CASE 1: BiBTeX used to constantly update the references
		%   (while the paper is being written).
		%\bibliographystyle{ormsv080} % outcomment this and next line in Case 1
		\bibliographystyle{plain}
		\bibliography{bibliography}
		%\input{paper.bbl}
		%\bibliography{paper.bbl} % if more than one, comma separated
		% CASE 2: BiBTeX used to generate mypaper.bbl (to be further fine tuned)
		%\input{mypaper.bbl} % outcomment this line in Case 2
		
		%If you don't use BiBTex, you can manually itemize references as shown below.

		%%%%%%%%%%%%%%%%%
	\end{document}